\documentclass[sn-mathphys-num]{sn-jnl}

\usepackage{graphicx}%
\usepackage{multirow}%
\usepackage{amsmath,amssymb,amsfonts}%
\usepackage{amsthm}%
\usepackage{mathrsfs}%
\usepackage[title]{appendix}%
\usepackage{xcolor}%
\newcommand{\red}[1]{\textcolor{black}{#1}}
\newcommand{\redd}[1]{\textcolor{black}{#1}}

\usepackage{textcomp}%
\usepackage{manyfoot}%
\usepackage{booktabs}%
\usepackage{algorithm}%
\usepackage{algorithmicx}%
\usepackage{algpseudocode}%
\usepackage{listings}%

\usepackage{accents}
\newcommand{\ubar}[1]{\underaccent{\bar}{#1}}

\usepackage{tkz-tab}
\usepackage{xpatch}
\usepackage{enumerate} 

\usepackage{slashbox}
\usepackage{adjustbox}
\usepackage{tabularx}

\usepackage{array} 

\usepackage{booktabs}

\usepackage{multicol}
\usepackage{comment}

\theoremstyle{thmstyleone}%
\newtheorem{theorem}{Theorem}
\theoremstyle{thmstyletwo}%
\newtheorem{remark}{Remark}%

\theoremstyle{thmstylethree}%
\newtheorem{definition}{Definition}%
\newtheorem{assumption}{Assumption}
\newtheorem{lemma}{Lemma}

\DeclareMathOperator{\crit }{crit}
\DeclareMathOperator{\dist}{dist}

\begin{document}

\title[Article Title]{Convergence rates for an inexact linearized ADMM for nonsmooth \redd{nonconvex} optimization with nonlinear equality constraints}


\author[1]{\fnm{Lahcen} \sur{El Bourkhissi}}\email{lahcenelbourkhissi1997@gmail.com}
\equalcont{These authors contributed equally to this work.}

\author*[1,2]{\fnm{Ion} \sur{Necoara}}\email{ ion.necoara@upb.ro}
\equalcont{The authors contributed equally to this work.}


\affil[1]{\orgdiv{Automatic Control and Systems Engineering Department}, \orgname{National University of Science and Technology Politehnica Bucharest}, \orgaddress{\street{} \city{} \postcode{060042}, \state{Bucharest}, \country{Romania}}}

\affil[2]{\orgdiv{Gheorghe Mihoc-Caius Iacob Institute of Mathematical Statistics and Applied Mathematics of the  Romanian Academy}, \orgname{} \orgaddress{\street{} \city{} \postcode{050711}, \state{Bucharest}, \country{Romania}}}



\abstract{In this paper, we consider  nonconvex  optimization problems with nonsmooth \redd{nonconvex} objective function and  nonlinear equality constraints. We assume that both the objective function and the functional constraints can be separated into 2 blocks. To solve this problem, we introduce a new inexact linearized alternating direction method of multipliers (ADMM) algorithm. Specifically, at each iteration, we linearize the smooth part of the objective function and the nonlinear part of the functional constraints within the augmented Lagrangian and add a dynamic quadratic regularization. We then compute the new iterate of the block associated with nonlinear constraints inexactly. This strategy yields subproblems that are easily solvable and their (inexact) solutions become the next iterates. Using Lyapunov arguments, we establish convergence guarantees for the iterates of our method toward an $\epsilon$-first-order  solution within  $\mathcal{O}(\epsilon^{-2})$  iterations. Moreover, we demonstrate that in cases where the problem data exhibit e.g., semi-algebraic properties or more general the KL condition, the entire sequence generated by our algorithm converges, and we provide convergence rates. To validate both the theory and the performance of our  algorithm, we conduct numerical simulations for  several nonlinear model predictive control and \redd{matrix factorization} problems.}




\maketitle

\section{Introduction}

\noindent Many applications such as nonlinear model predictive control,  state and parameter estimation of dynamical systems,  training shallow neural networks, classification and signal processing can be formulated as the following structured  nonsmooth nonconvex optimization problem with nonlinear equality constraints  of the form (see for example  \cite{Fes:20,HonHaj:17,GamKiz:24,MesBau:21,CohHal:21}): 
\begin{equation}
\begin{aligned}\label{eq1}
& \underset{x\in\mathbb{R}^n, y \in {\mathcal{Y}} \subseteq\mathbb{R}^p}{\min}
& & f(x)+g(x)+h(y)\\
& \hspace{0.7cm}\textrm{s.t.:}
& & \hspace{0.07cm} F(x)+Gy=0,
\end{aligned}
\end{equation}
where {$\mathcal{Y}$ is a nonempty, closed  subset of $\mathbb{R}^p$ which admits an easy projection, the matrix $G \in \mathbb{R}^{m \times p}$, the functions   $f:\mathbb{R}^n\to {\mathbb{R}}, g:\mathbb{R}^n\to \bar{\mathbb{R}}, h:\mathbb{R}^p\to {\mathbb{R}},$  and  $F\triangleq{(f_1,...,f_m)}^T$,  
 with $f_i:\mathbb{R}^n\to {\mathbb{R}}$ for all $i \in\{1,...,m\}$, are nonlinear functions. We consider that $f, h, f_i$,  for all  $i=1,...,m$, are continuously differentiable functions, $f, h$ are possibly nonconvex and $g$ is a \redd{proper lower semi-continuous and prox-bounded function relative to its domain $\text{dom}g$ (possibly nonconvex)}.   Moreover, we assume that the problem is well-posed, i.e., the feasible set is nonempty and the optimal value is finite. \redd{Note that problem \eqref{eq1} is more general than the one considered e.g., in \cite{CohHal:21}; specifically, \cite{CohHal:21} can have constraints only on  the block variables $x$,  while we can impose constraints on  both block variables $x$ and $y$, respectively.} Moreover, inequality constraints on $x$ can be included in the function $g$ using the indicator function.  For example, any constrained  composite optimization problem frequently appearing in optimal control~\cite{MesBau:21}:
\begin{align*}
\min_{x \in \mathcal{X}} f(x)+h(F(x)) \quad \text{s.t. }\; F(x)\in\mathcal{Y},
\end{align*} can be easily recast in the form of optimization problem \eqref{eq1} by defining $F(x)=y$,  then $G = -I_m$ and $g$ the indicator function of the set $\mathcal{X}$. 

\medskip

\noindent \textit{Related work.}  In this paper, we propose an augmented Lagrangian approach to address problem \eqref{eq1}. The augmented Lagrangian method, or method of multipliers, was introduced in \cite{HaaBuy:70,Hes:69} to minimize an objective function under equality constraints. It provides many theoretical advantages, even for non-convex problems (e.g., no duality gap and exact penalty representation), see \cite{RocWet:98}.  In constrained optimization, the augmented Lagrangian approach has attracted considerable attention and has been widely studied for convex problems, for example in \cite{Ber:96,BoyPar:11,NedNec:14,PatNec:17, HeYua:13} and related references. More recently, researchers have extended this approach to non-convex problems, encompassing both smooth and non-smooth objectives with linear constraints, see for example \cite{ThemPat:20, Yas:22,QinXin:19,JiaLin:19,ZhaLuo:20,HonHaj:17,HajHon:19}. However, there are  very few studies on  the use of the augmented Lagrangian framework for nonconvex optimization, where nonconvexity comes from  constraints, e.g.,   \cite{XieWri:21,CohHal:21,KreMar:00, BirMar:14}.  In particular, in \cite{XieWri:21}, a proximal augmented Lagrangian (Proximal AL) algorithm is proposed to solve the problem \eqref{eq1}; in this method, a static proximal term is added to the original augmented Lagrangian function. It is proved that when an approximate first- (second-) order solution of the subproblem is found, then an $\epsilon$ first-  (second-) order solution of the original  problem \eqref{eq1} is obtained within $\mathcal{O}(\epsilon^{\eta-2})$ outer iterations, for some parameter  $\eta\in[0,2]$. Note that when $\eta$ is  close to $2$, the efficiency is reduced to $\mathcal{O}(1)$ outer iterations, but the subproblem, which is already non-convex, becomes very ill-conditioned as the penalty parameter of the augmented Lagrangian is inversely proportional to $\epsilon^{\eta}$. 

\medskip 

\noindent On the other hand, when the optimization problem at hand possesses a specific structure, such as the separability discussed in this paper, it becomes feasible to leverage this inherent structure within the augmented Lagrangian framework. This approach leads to the well-known Alternating Direction Method of Multipliers (ADMM) (refer to \cite{BoyPar:11, GabMer:76, GloTal:89, CohHal:21, Yas:22,  BolSab:18, BotNgu:20, HagZha:20}). In particular,  in  \cite{BoyPar:11}, an ADMM method was introduced for solving a separable convex problem with linear constraints, and its asymptotic convergence was demonstrated. In \cite{HagZha:20}, the authors established the convergence rate of an inexact ADMM, designed for solving a nonsmooth convex problem with linear constraints. It was proven that when the objective function is convex their method \redd{has} complexity   $\mathcal{O}(\epsilon^{-1})$  and  $\mathcal{O}(\epsilon^{-0.5})$ in the strongly convex setting. Furthermore, \cite{Yas:22} proposed a Proximal Linearized ADMM (PL-ADMM) to address nonsmooth nonconvex problems with linear equality constraints for two blocks. In this scenario, one block of the problem is smooth, and the other is nonsmooth. The PL-ADMM algorithm linearizes the smooth parts during each block update, incorporates a dynamic metric proximal term in the primal update, and includes an overrelaxation step in the dual update. Paper \cite{Yas:22} demonstrated that each limit point of the sequence generated by this algorithm is a stationary point of the original problem. Additionally, it was shown that under the Kurdyka-Lojasiewicz  (KL) property, the method converges either in a finite number of iterations, sublinearly, or linearly, depending on the exponent associated with the KL condition. 

\medskip 

\noindent However, for nonconvex problems with nonlinear equality constraints, limited research has been conducted, specifically in \cite{CohHal:21, BolSab:18, DemJia:23}. In \cite{BolSab:18}, the authors addressed a special case of problem \eqref{eq1} ($g=0$, $\mathcal{Y} = \mathbb{R}^n$ and $G=-I$, where $I$ denotes the identity matrix of appropriate dimension), proving that when the primal iterates are approximate stationary points of the augmented Lagrangian function, each limit point  is a first-order solution of the problem. Furthermore, it was demonstrated that, under KL, the entire sequence generated by their method converges to a first-order solution of the problem.   Further, in \cite{CohHal:21}, the authors proposed an augmented Lagrangian-based method to tackle the  problem considered in this paper (with $\mathcal{Y} = \mathbb{R}^n$) and under similar assumptions as ours. In this algorithm, the authors linearized the smooth part of the augmented Lagrangian function and added a dynamic quadratic regularization term,  proving the convergence of the iterates of their method to a KKT point of problem \eqref{eq1} under the KL property. \redd{Furthermore, \cite{HalTeb:23} considered nonconvex composite functional optimization problems. To handle this structure, the authors introduced a slack variable, leading to a problem of the form  \eqref{eq1}, where the function associated with the slack variable is  lower semicontinuous and $G$ is given by $-I_m$. In contrast, in our case $h$ is smooth and $\mathcal{Y}$ satisfies Assumption \ref{assum_lips}, while $G$ is a general full row rank matrix (see Section 2).  The authors in \cite{HalTeb:23} then applied an ADMM scheme to solve the reformulated problem, where each iteration involves linearizing the smooth component of the augmented Lagrangian function, leading to proximal updates at each step of the algorithm. Under the assumption that the dual multipliers remain bounded (as also assumed in this paper), their method achieves an $\epsilon$-first-order  solution within $\mathcal{O}(\rho^2/\epsilon^{2})$ iterations. However, to ensure that the iterates remain $\epsilon$-feasible after a certain number of iterations, the penalty parameter  must satisfy $\rho \geq \mathcal{O}(\epsilon^{-1})$. Consequently, based on the proof of Theorem 2 in \cite{HalTeb:23}, the overall complexity of their method ultimately reaches $\mathcal{O}(\epsilon^{-4})$.}
\redd{Finally, in \cite{DemJia:23}, the authors considered a special case of problem \eqref{eq1} ($h=0$, $G=-I$, while $\mathcal{Y}$ is a general set). The authors proposed an augmented Lagrangian method with a variable penalty parameter proving that when the primal iterates are approximate stationary points of the augmented Lagrangian function and the dual updates are projected onto some compact set, each limit point is a stationary point of the feasible problem and no connection to KKT points could be established. }


\medskip

\noindent \textit{Contributions:} Our approach, the inexact linearized ADMM, addresses several key limitations present in previous works. Notably, in \cite{XieWri:21, BolSab:18, DemJia:23}, high computational costs are required for solving  the nonconvex subproblems. In \cite{BoyPar:11, HagZha:20, Yas:22}, only linear constraints were considered. Furthermore, in \cite{CohHal:21}, no information about the Hessian of the augmented term is used, resulting in a low-quality approximation, and no constraints on the second block variables  are allowed. \redd{Finally, although \cite{DemJia:23} considers a general set $\mathcal{Y}$, this leads to weak convergence results.}  Specifically, our main  contributions are:

\vspace{0.1cm}

\noindent \textbf{(i)} At each iteration, we linearize the smooth part of the cost function and the nonlinear part of the functional constraints in the augmented Lagrangian function. Additionally, we introduce a dynamic regularization term. Furthermore, we solve the block associated with nonlinear constraints inexactly. This gives rise to a new algorithm, named the inexact Linearized ADMM (iL-ADMM) method, which, at each iteration, requires solving simple subproblems that are easy to address.

\vspace{0.1cm}

\noindent \textbf{(ii)}  {Nonlinearity in the constraints related to one block of variables in problem \eqref{eq1} introduces nontrivial challenges compared to the linear constraints case. More specifically, the usual ADMM algorithm developed originally for linear constraints  must be properly modified and consequently a new convergence analysis is required.} We provide rigorous proofs, based on Lyapunov function arguments, of global asymptotic convergence, proving that the iterates converge to a critical point of the augmented Lagrangian function. Additionally, our method guarantees convergence to an $\epsilon$-first-order  solution of the original problem in at most $\mathcal{O}(\epsilon^{-2})$~iterations. 

\vspace{0.1cm}

\noindent \textbf{(iii)} Under the (KL) property, which holds e.g., for semi-algebraic functions, we demonstrate the convergence of the entire sequence generated by our algorithm and derive improved convergence rates that depend on the KL parameter. 

\medskip 

\noindent In comparison with \cite{XieWri:21}, our approach effectively utilize the unique structure of the problem, particularly its separability. When comparing the complexity of the subproblems, the algorithms in \cite{BolSab:18, XieWri:21} is difficult to implement in practice due to their highly nonconvexity caused by the presence of nonlinear constraints in the subproblem from each iteration. Moreover, unlike \cite{BoyPar:11, HagZha:20, Yas:22}, our iL-ADMM method can handle nonlinear equality constraints. Furthermore, unlike \cite{CohHal:21}, our method uses a Gauss-Newton approach to retain some information about the Hessian of the augmented term making use of  only first-order derivatives. Finally, we compare the efficiency of our method with IPOPT \cite{WacBie:06} and the augmented Lagrangian  method in \cite{CohHal:21} to solve  nonlinear model predictive control and \redd{matrix factorization} problems using real systems and datasets, respectively. This paper represents an extension of our earlier work \cite{ElbNec:23}. The extension involves examining the impact of inexactness in solving the subproblem associated with the first block of the primal variables, diverging from the exact solution approach in \cite{ElbNec:23}. Furthermore, we delve into the convergence rate analysis under general or KL conditions, in contrast to \cite{ElbNec:23}, which focused solely on proving asymptotic convergence. Additionally, our study incorporates more  numerical~tests.

\medskip

\noindent The paper is structured as follows. In the next section, we introduce some mathematical preliminaries, in section \ref{sec3} we present the  iL-ADMM method followed in section \ref{sec4} by its convergence analysis. Finally, section \ref{sec5} presents detailed numerical results.

\section{Preliminaries}\label{sec2}
\noindent We use $\|\cdot\|$ to denote the $2-$norm of a vector or of a matrix, respectively.  For a differentiable function $\phi:\mathbb{R}^l\to\mathbb{R}$, we denote by $\nabla \phi(x)\in\mathbb{R}^l$ its gradient at a point $x$. For a differentiable function $F:\mathbb{R}^n  \to\mathbb{R}^m$, we denote its Jacobian at a given point $x$ by $\nabla F(x)\in\mathbb{R}^{m\times n}$. Furthermore, $\partial g(x)$ refers to the limiting subdifferential of a proper, lsc function $g$, and $\partial^{\infty} g(x)$ refers to the horizon subdifferential. For more details about the subdifferential of nonsmooth nonconvex functions, we refer  to  \cite{RocWet:98}. 
{Moreover, $N_{\mathcal{Y}}(y)$ denotes the normal cone at $y \in \mathcal{Y}$ associated to the set $\mathcal{Y}$ \redd{and \( \mathbb{Z}_{+} \) is used to denote the set of positive integers}.}
\noindent We further introduce the  following notations:
\begin{gather*}
     l_f(x;\bar{x}):=f(\bar{x})+\langle\nabla f(\bar{x}),x-\bar{x}\rangle\hspace{0.5cm}\forall x,\bar{x},\\
     l_h(y;\bar{y}):=h(\bar{y})+\langle\nabla h(\bar{y}),y-\bar{y}\rangle\hspace{0.5cm}\forall y,\bar{y},\\
     l_F(x;\bar{x}):=F(\bar{x})+\nabla F(\bar{x})(x-\bar{x}) \hspace{0.5cm}\forall x,\bar{x}.
\end{gather*}
\redd{
A point $(x^*,y^*) \in \mathbb{R}^n \times \mathbb{R}^p$ is said to be feasible for \eqref{eq1} if $(x^*,y^*)\in\text{dom }g \times \mathcal{Y}$ 
 and $F(x^*) + Gy^*=0$. 
Let us introduce the definition of a KKT point of problem \eqref{eq1} and  motivate why we are interested in designing algorithms that yield such points.
\medskip
\begin{definition}\label{firstorder}[KKT and $\epsilon$-KKT points]
 The vector $(x^*,y^*) \in \text{dom }g\times \mathcal{Y}$ is said to be  a KKT point of problem \eqref{eq1} (equivalently,  first-order solution of problem \eqref{eq1}) if  $\exists \lambda^*\in\mathbb{R}^m$ such that the following conditions are satisfied:
\begin{gather*}
    -\nabla f(x^*)-{\nabla F(x^*)}^T\lambda^*\in\partial g(x^*), \quad 0\in\nabla h(y^*)+G^T\lambda^*+N_{\mathcal{Y}}(y^*), \\
 F(x^*)+Gy^*=0. 
\end{gather*}
Moreover, let $\epsilon>0$. The vector $(x_{\epsilon}^*,y_{\epsilon}^*) \in \text{dom }g\times \mathcal{Y}$ is said to be  an $\epsilon$-KKT point of problem \eqref{eq1} (equivalently, $\epsilon$-first-order solution of problem \eqref{eq1}) if  $\exists \lambda_\epsilon^*\in\mathbb{R}^m$ such that the following conditions are satisfied:
\begin{gather*}
    \dist\left(-\nabla f(x_{\epsilon}^*)-{\nabla F(x_{\epsilon}^*)}^T\lambda_{\epsilon}^*, \partial g(x_{\epsilon}^*)\right) \leq \epsilon, \quad \dist\left(-\nabla h(y_{\epsilon}^*) - G^T\lambda_{\epsilon}^*,N_{\mathcal{Y}}(y_{\epsilon}^*)\right) \leq \epsilon, \\
 \|F(x_{\epsilon}^*)+Gy_{\epsilon}^*\|\leq\epsilon. 
\end{gather*}
\end{definition}
\noindent Next, we present  a constraint qualification  condition for \eqref{eq1} at a feasible point $(\bar{x},\bar{y}) $.
\medskip 
\begin{definition} \label{regularity_def}[Constraint Qualification (CQ)]
Let $(\bar{x},\bar{y})$ be a feasible point for problem \eqref{eq1}. The pair $(\bar{x},\bar{y})$  is regular if the following constraint qualification conditions hold:
\begin{enumerate}[(i)]
    \item $\partial^\infty g(\bar{x}) \cap \text{range} \, \left(\nabla F(\bar{x})^T\right)=\{0\}$.
    \item  $N_{\mathcal{Y}}(\bar{y}) \cap \text{range} \, \left(G^T\right)=\{0\}$.
\end{enumerate}
\end{definition}
\medskip 
\noindent Next lemma shows why KKT points are interesting to us (its proof can be found in \cite{KruMeh:22}(Proposition 6.9)). 
\begin{lemma} \label{lemma_KKT}[Local minima + CQ = KKT points]
     Let $(x^*, y^*)\in \text{dom }g\times \mathcal{Y}$ be a local minimizer of problem \eqref{eq1}, which satisfies CQ. Then, $(x^*, y^*)$ is a KKT point of  \eqref{eq1}.
\end{lemma}
}

\medskip 

\noindent \redd{Lemma \ref{lemma_KKT} shows that the KKT conditions are necessary for local minimizers, provided that some constraint qualification conditions hold at those points. This is the primary motivation for designing algorithms that converge to KKT points.  The convergence analysis of this paper will show that the limit points of the iterates generated  by our algorithm proposed in Section \ref{sec3} are KKT points, which, according to the previous lemma, may be  local minima, provided that the constraint qualification from Definition \ref{regularity_def} holds at such points. However, for the convergence analysis conducted in the next sections, we only consider  the following assumptions:}

\medskip 

\begin{assumption}\label{assump2}
Given the compact sets $\mathcal{S}_x\subseteq \text{dom }g$ and $\mathcal{S}_y\subseteq \mathcal{Y}$, there exist positive constants $ \sigma, L_f, L_h, L_F$ such that the functions  $f, h$ and $F$ satisfy the following conditions  for all  $x, x'\in\mathcal{S}_x$  and  for all $y, y'\in\mathcal{S}_y$:
\begin{enumerate}[(i)]
  \item$ \|\nabla f(x)-\nabla f(x')\|\leq L_f\|x-x'\|,$  \label{ass_ii}
  \item$ \|\nabla h(y)-\nabla h(y')\|\leq L_h\|y-y'\|,$  \label{h_ass_ii}
  \item $    \|\nabla F(x)-\nabla F(x')\|_2\leq L_F\|x-x'\|,$ \label{ass_iv}
  \item $\sigma_{\text{min}}(G)\geq\sigma$. \label{ass_v}
\end{enumerate}
\end{assumption}

\noindent An immediate  consequence of the above assumption is that if Assumption \ref{assump2} is satisfied on the compact sets $\mathcal{S}_x, \mathcal{S}_y$, then there exist positive constants $ M_h, M_F$ such that the functions  $ h$ and $F$ satisfy the following conditions  for all  $(x, y)\in\mathcal{S}_x\times\mathcal{S}_y$ (a given compact set):
\begin{equation}\label{assump_conseq}
    \|\nabla h(y)\|\leq M_h,  \;  \|\nabla F(x)\|_2\leq M_F.
\end{equation}

\medskip 

\noindent Note that these assumptions are standard in  nonconvex optimization, see e.g., \cite{XieWri:21,CohHal:21}.
    In fact, it covers a large class of problems; more precisely,  \textit{\eqref{ass_ii}} and \textit{\eqref{h_ass_ii}} hold if $f$ and $h$ are  smooth on a neighborhood of $\mathcal{S}_x$ and  $\mathcal{S}_y$, respectively. Similarly,  \textit{\eqref{ass_iv}} is valid if $F$ is smooth  on a neighborhood of ${\mathcal{S}_x}$. Note that these assumptions are not very restrictive because they are satisfied locally for any $f, h, F\in\mathcal{C}^2$.  Moreover, assumption \textit{\eqref{ass_v}} is equivalent to the matrix \( G \) having full row rank. {In addition, the structure in the  problem \eqref{eq1} allows us to make more relaxed assumptions compared to problems without any linear part in the functional constraints. For example, instead of assuming that the Jacobian of the nonlinear functional constraints satisfies the Linear Independence Constraint Qualification (LICQ) on a given set, which is difficult to check, (see e.g.,  \cite{XieWri:21}),  our assumption \textit{\eqref{ass_v}} asks a simple  and easily verifiable condition only on the matrix $G$ corresponding to the  block of the decision variables $y$.

\medskip 

\noindent For this structured problem, one can develop pure augmented Lagrangian-based algorithms with simple subproblems (see our algorithm below).} We further introduce the following notations:
\[
\psi_{\rho}(x,y,\lambda)=f(x)+\langle \lambda ,F(x) +Gy\rangle+\frac{\rho}{2}{\|F(x)+Gy\|^2}.
\]
The gradient of $\psi_{\rho}$ is given by:
\[
\begin{cases}
     \nabla_x\psi_{\rho}(x,y,\lambda)=\nabla f(x)+{\nabla F(x)}^T\left(\lambda+\rho \left(F(x)+Gy\right)\right),\\
      \nabla_y\psi_{\rho}(x,y,\lambda)={G}^T\left(\lambda+\rho \left(F(x)+Gy\right)\right),\\
      \nabla_{\lambda}\psi_{\rho}(x,y,\lambda)=F(x)+Gy.
\end{cases}
\]
\begin{remark}
Note that if Assumption \ref{assump2} holds on a compact set $\mathcal{S}_x\times\mathcal{S}_y\subseteq \text{dom }g\times\mathcal{Y}$, then for any ball  \(\mathbb{B}_r \subset \mathbb{R}^m\)  centered at zero with radius \(0 \leq r < \infty\), the function  $\psi_{\rho}$  has  Lipschitz continuous gradient on the compact set $\mathcal{S}_x\times\mathcal{S}_y\times\mathbb{B}_r$ with Lipschitz constant (see Lemma 4.1 in \cite{CohHal:21}):
\begin{align*}
L_\psi=L_f +  \!\! \sup_{(x,y,\lambda)\in\mathcal{S}_x\times\mathcal{S}_y\times\mathbb{B}_r}  \!\!  \{L_F \|\lambda + \rho(F(x)+Gy)\|\}  + (M_F + \|G\|)(2 + \rho(M_F +\|G\|)).
\end{align*}
\end{remark}

\noindent \redd{  Let us now define the Hausdorff distance between two bounded sets, which will be used in our convergence analysis.
\begin{definition}[Hausdorff distance \cite{Don:21}]  
Given two bounded sets \( A , B \subset \mathbb{R}^p\), the Hausdorff distance between \( A \) and \( B \) is defined as:  
\[
\dist_H(A, B) := \max \left\{ \sup_{a \in A} \dist(a, B), \ \sup_{b \in B} \dist(A, b) \right\},
\]
where \(\dist(a, B) = \inf_{b \in B} \|a - b\|\) and \(\dist(A, b) = \inf_{a \in A} \|a - b\|\).
\end{definition}
\noindent Let \(\mathcal{Y} \subseteq \mathbb{R}^p\) be a nonempty closed set.   We denote a space of  nonempty bounded subsets of the normal cones of \(\mathcal{Y}\) as follows: 
\[
\Sigma_{\mathcal{Y}} := \left\{\bar{N}_{\mathcal{Y}}(y) \mid  \bar{N}_{\mathcal{Y}}(y) \subset N_{\mathcal{Y}}(y), \bar{N}_{\mathcal{Y}}(y) \; \text{bounded},  \;  y \in \mathcal{Y} \right\},
\]
where \(N_{\mathcal{Y}}(y)\) is the normal cone to \(\mathcal{Y}\) at \(y\). 
For example, for a general set $\mathcal{Y}$ and for a given $r>0$, we can define: 
\begin{equation} \label{normal1}
    \bar{N}_{\mathcal{Y}}(y) = N_{\mathcal{Y}}(y) \cap \mathbb{B}_r \quad \forall y\in \mathcal{Y},
\end{equation}
where \(\mathbb{B}_r \subset \mathbb{R}^p\) denotes the ball centered at zero with radius \(r\),  and then:
\[
\Sigma_{\mathcal{Y}} := \left\{  N_{\mathcal{Y}}(y) \cap \mathbb{B}_r \mid  y \in \mathcal{Y} \right\}. 
\]
Moreover, if  $\mathcal{Y}$ is a differentiable manifold, i.e., $\mathcal{Y} := \{y\in\mathbb{R}^p | H(y)=0\}$, where $H: \mathbb{R}^p\to \mathbb{R}^q$ is a continuously differentiable function with its Jacobian having full row rank on $\mathcal{Y}$, then its normal cone is described by $N_{\mathcal{Y}}(y) = \{w\in \mathbb{R}^p | w=\nabla H(y)^T\lambda, \lambda\in\mathbb{R}^q\}$ and we can define: 
\begin{equation}\label{normal2}
    \bar{N}_{\mathcal{Y}}(y) =  \{w\in \mathbb{R}^p | w=\nabla H(y)^T\lambda, \lambda\in \mathbb{B}_r \} \quad \forall y\in \mathcal{Y},
\end{equation}
and then:
\[
\Sigma_{\mathcal{Y}} := \left\{   \{\nabla H(y)^T\lambda, \lambda\in \mathbb{B}_r\} \mid \;  y\in\mathcal{Y}\right\}.
\]
 We equip such a space \(\Sigma_{\mathcal{Y}}\) with the  metric given by the Hausdorff distance.  Then, in the sequel, we also impose  the following restriction on the set $\mathcal{Y}$:
\begin{assumption}\label{assum_lips}
We assume that the set $\mathcal{Y}$ admits a Lipschitz continuous normal cone  mapping, i.e., for  $\bar{N}_{\mathcal{Y}}(\cdot)$ either of the form \eqref{normal1} or \eqref{normal2}, there exists  $\kappa>0$ such that:
    \[
\dist_H\left(\bar{N}_{\mathcal{Y}}(y), \bar{N}_{\mathcal{Y}}(y')\right)) \leq \kappa \|y - y'\| \quad \forall y, y' \in \mathcal{Y}.
\]
\end{assumption}
\noindent Below, we provide some examples when Assumption \ref{assum_lips} holds.
\begin{lemma}\label{claim1}[Sets with Lipschitz continuous normal cones]  
For any set of the form \(\mathcal{Y}:=\left\{y \in\mathbb{R}^p | H(y) = 0 \right\}\), where $H: \mathbb{R}^p \to \mathbb{R}^q$ has the Jacobian, $\nabla H$, Lipschitz continuous and,  additionally, it is of full row rank on $\mathcal{Y}$, the normal cone mapping  $y \rightrightarrows \bar{N}_{\mathcal{Y}} =  \{w\in \mathbb{R}^p | w=\nabla H(y)^T\lambda, \lambda\in \mathbb{B}_r \}$  is Lipschitz continuous for any $r>0$. In particular, the full space $\mathbb{R}^p$, any affine subspace or sphere  yield Lipchitz continuous normal cone mappings  $\bar{N}_{\mathcal{Y}}(\cdot)$, for both  \eqref{normal1} and \eqref{normal2}.
\end{lemma}
\begin{proof}
 Let $H$ have the Jacobian Lipschitz continuous with Lipschitz constant $L_H>0$ and, additionally,  the Jacobian has full row rank on $\mathcal{Y}$. Then, the tangent and normal cones of $\mathcal{Y}$ at a given point $y\in\mathcal{Y}$,  $T_{\mathcal{Y}}(y)$ and $ N_{\mathcal{Y}}(y)$, are defined  as:
\[
T_{\mathcal{Y}}(y) = \left\{ v\in \mathbb{R}^p | \;  \nabla H(y) v = 0\right\} \quad \text{and} \quad   N_{\mathcal{Y}}(y) = \left\{ w\in \mathbb{R}^p | \;  w = \nabla H(y)^T \lambda, \;  \lambda \in \mathbb{R}^q\right\},
\]
respectively. Further,  we have:
\begin{align*}
   & \dist_H(\bar{N}_{\mathcal{Y}}(y), \bar{N}_{\mathcal{Y}}(y'))  \\
    &= \max\left\{\sup_{a\in \bar{N}_{\mathcal{Y}}(y) }\dist_H\left( a,\bar{N}_{\mathcal{Y}}(y')\right), \sup_{b\in \bar{N}_{\mathcal{Y}}(y') }\dist_H\left( \bar{N}_{\mathcal{Y}}(y), b \right) \right\}. 
\end{align*}
Since the Jacobian has full row rank and $r>0$, we have from \eqref{normal2}:
\begin{align*}
    &\sup_{a\in \bar{N}_{\mathcal{Y}}(y) }\dist_H\left( a,\bar{N}_{\mathcal{Y}}(y') \right) \\
    &= \max_{\lambda\in\mathbb{B}_r} \min_{\theta\in\mathbb{B}_r}\|\nabla H(y)^T\lambda - \nabla H(y')^T\theta\|\\
    &=\max_{\lambda\in\mathbb{B}_r}\min_{\theta\in\mathbb{B}_r}\|\left(\nabla H(y)- \nabla H(y')\right)^T \lambda + \nabla H(y')^T(\lambda-\theta)\|\\
    & \leq \max_{\lambda\in\mathbb{B}_r} \left(\|\left(\nabla H(y)- \nabla H(y')\right)^T\lambda\| +\min_{\theta\in\mathbb{B}_r} \|\nabla H(y')^T(\lambda-\theta)\| \right)\\
    & = \max_{\lambda\in\mathbb{B}_r}\|\left(\nabla H(y)- \nabla H(y')\right)^T\lambda\|   \\
    & \leq \|\nabla H(y)- \nabla H(y')\|\max_{\lambda\in\mathbb{B}_r}\|\lambda\|   \\
    & \leq L_H \|y-y'\| r = \kappa \|y-y'\|,
\end{align*}
where $\kappa = L_H \times r$. Similarly, by simmetry arguments,  we have:
\begin{align*}
    \sup_{b\in \bar{N}_{\mathcal{Y}}(y')}\dist_H\left( \bar{N}_{\mathcal{Y}}(y), b\right) \leq \kappa \|y-y'\|.
\end{align*}
Hence, there exists $\kappa=r\times L_H >0$  such that:
    \[
\dist_H\left(\bar{N}_{\mathcal{Y}}(y), \bar{N}_{\mathcal{Y}}(y')\right)) \leq \kappa \|y - y'\| \quad \forall y, y' \in \mathcal{Y}.
\]
Since affine subspaces and spheres are special cases of the set $\mathcal{Y}= \{y | H(y)=0\}$, from the proof above it follows that the corresponding normal cone mapping $\bar{N}_{\mathcal{Y}}(y)$ defined by \eqref{normal2} is Lipschitz continuous. On the other hand, for $\bar{N}_{\mathcal{Y}}(y)$ defined in \eqref{normal1}, we have:
\begin{align*}
    & \text{full space } (\mathbb{R}^p): \; \bar{N}_{\mathcal{Y}}(y) = {N}_{\mathcal{Y}}(y) =  \{ 0 \},\\
    & \text{affine subspace } (Ay=b): \; \bar{N}_{\mathcal{Y}}(y) = \{A^T\lambda \mid  \lambda\in\mathbb{R}^q \}\cap \mathbb{B}_r,\\
    & \text{sphere } (\|y\|^2=1): \; \bar{N}_{\mathcal{Y}}(y) = \{\lambda y \mid  \lambda\geq 0\} \cap \mathbb{B}_r,
\end{align*}
and then  using similar argument as above, the second claim follows.
\end{proof}
 \noindent The reader may find other examples of sets $\mathcal{Y}$ satisfying Assumption \ref{assum_lips}.
}
\noindent  Let  us  also introduce the  Kurdyka-Lojasiewicz (KL) property, which will lead to improvements in the convergence rates of our algorithm.  Let	 $\Phi:\mathbb{R}^d\to\bar{\mathbb{R}}$ be a proper lsc function. For $-\infty<\tau_1<\tau_2\leq
+\infty$, we define $[\tau_1<\Phi<\tau_2]=\{x \in\mathbb{R}^d :\tau_1<\Phi(x)<\tau_2\}$. 
Denote  $\Psi_{\tau}$ the set of all continuous concave functions $\varphi: [0, \tau] \to [0,+\infty)$ such
that $\varphi(0) = 0$ and $\varphi$ is continuously differentiable on $(0, \tau)$, with $\varphi'(s) > 0$ over $(0, \tau)$.

\begin{definition} \label{def2}
Let $\Phi : \mathbb{R}^d \to \bar{\mathbb{R}}$  be a proper lower semicontinuous function that takes constant value on a set $\Omega \subseteq \mathbb{R}^d$. We say that $\Phi$ satisfies the KL property on $\Omega$ if there exists $ \epsilon>0, \tau>0$, and $\varphi\in\Psi_{\tau}$ such that for every
$x^* \in \Omega$ and any  $x$ in the intersection $\{x\in\mathbb{R}^d: \text{ dist}(x,\Omega)<\epsilon\}\cap[\Phi(x^*)<\Phi(x)<\Phi(x^*)+\tau]$, we have:
\[
    \varphi'\big(\Phi(x) - \Phi(x^*)\big)\dist\big(0, \partial\Phi(x)\big) \geq1.
\] 
\end{definition}

\noindent  The KL property holds for a large class of functions including semi-algebraic functions (e.g., real polynomial functions), vector or matrix (semi)norms (e.g., $\|\cdot\|_p$ with $p \geq 0$ rational number), logarithm functions,  exponential functions and  uniformly convex functions,  see \cite{BolDan:07} for a comprehensive list. For the rest of this paper, we use the following notation:
\[
 l_{\psi_{\rho}}(x,\bar{y},\bar{\lambda};\bar{x}):=\psi_{\rho}(\bar{x},\bar{y},\bar{\lambda})+\langle\nabla \psi_{\rho}(\bar{x},\bar{y},\bar{\lambda}\redd{)},x-\bar{x}\rangle\hspace{0.5cm}\forall x,\bar{x}.
\]
The augmented Lagrangian associated with \eqref{eq1} is:
\begin{align*}
 \mathcal{L}_{\rho}(x,y,\lambda)&=f(x)+g(x)+h(y)+\langle \lambda,F(x)+Gy \rangle+\frac{\rho}{2}{\|F(x)+Gy\|^2}\\
  &=g(x)+h(y)+\psi_{\rho}(x,y,\lambda).  
\end{align*}
Given a pair $(\bar{x},\bar{y})$, we introduce the following  linearized augmented Lagrangian:
\begin{align*}
 \bar{\mathcal{L}}_{\rho}(x,y,\lambda;\bar{x},\bar{y})=l_f(x;\bar{x})+g(x)+l_h(y;\bar{y}) +\langle {\lambda} ,l_F(x;\bar{x})+Gy \rangle+\frac{\rho}{2}{\|l_F(x;\bar{x})+Gy\|^2}.   
\end{align*}
Note that this approximation, $\bar{\mathcal{L}}_{\rho}$, of the true augmented Lagrangian, ${\mathcal{L}}_{\rho}$, retains  curvature information from $F$ through the term $\nabla F^T \nabla F$.

\section{A new Inexact Linearized ADMM algorithm}\label{sec3}
\noindent \redd{In this section, we propose an augmented Lagrangian-based method (Algorithm \ref{alg1}), which shares similarities with the approach introduced in \cite{CohHal:21}, albeit featuring a distinctive update for the primal variables (refer to Steps 4 and 5 below). While \cite{CohHal:21} linearizes the smooth part of the augmented Lagrangian function, $\psi_{\rho}$, with respect to $x$ and incorporates  a dynamic quadratic regularization term, our methodology adopts a Gauss-Newton type approach; linearizing the nonlinear functional constraint $F$ within $\psi_{\rho}$. This choice enhances the accuracy of our model's approximation to the original augmented Lagrangian function compared to the method in \cite{CohHal:21}. The reason being, the linearization technique used in \cite{CohHal:21} neglects curvature information about the nonlinear constraints $F$, whereas our algorithm leverages (partial) curvature information from $F$ through the term $\nabla F^T \nabla F$ present in $\bar{\mathcal{L}}_{\rho}$.
This improved approximation of the augmented Lagrangian not only provides theoretical advantages but also demonstrates practical implications, as demonstrated in our numerical simulations. Furthermore, in our approach, we solve the subproblem in Step 4 inexactly, diverging from \cite{CohHal:21}, where an exact solution is sought. Additionally, the regularization in Step 5 is dynamically chosen in our case, in contrast to \cite{CohHal:21}, where it is static, and additionally, we allow explicit constraints on $y$.}

\begin{algorithm}
\caption{Inexact Linearized ADMM (iL-ADMM)}\label{alg1}
\begin{algorithmic}[1]
\State  $\textbf{Initialization: } x_0, y_0, \lambda_0 \; \text{and} \;\rho, \theta_0, \beta_0, \alpha>0$
\State $k \gets 0$
\While{$\text{ stopping criterion is not satisfied }$}
    \State $\text{generate a proximal parameter } \beta_{k+1}\geq\beta_0$ such that
  $$x_{k+1} \approx \arg\min_{x}{\bar{\mathcal{L}}_{\rho}(x,y_{k},\lambda_{k};x_{k},y_{k})+\frac{\beta_{k+1}}{2}{\|x-x_{k}\|}^2}$$ satisfies an inexact stationary condition and a descent:
  \begin{equation*}
      \exists s_{k+1}\in\partial_x\left(\bar{\mathcal{L}}_{\rho}(x,y_{k},\lambda_{k};x_{k},y_{k})+\frac{\beta_{k+1}}{2}{\|x-x_{k}\|}^2\right)\redd{\bigg|_{x=x_{k+1}}}
  \end{equation*}
         \text{such that}
         \begin{equation}
        \|s_{k+1}\|\leq\alpha\|x_{k+1}-x_k\|\label{inexactness},
         \end{equation}
  \begin{equation}
    \psi_{\rho}(x_{k+1},y_{k},\lambda_{k})-l_{\psi_{\rho}}(x_{k+1},y_{k},\lambda_{k};x_k)\leq\frac{\beta_{k+1}}{4}\|x_{k+1}-x_{k}\|^2.\label{eq_assu}
\end{equation}
    \State $\text{generate a proximal parameter } \theta_{k+1}\geq\theta_0 $ such that
 $$ y_{k+1}\gets\arg\min_{y {\in\mathcal{Y}}}{\bar{\mathcal{L}}_{\rho}(x_{k+1},y,\lambda_{k};x_{k+1},y_{k})+\frac{\theta_{k+1}}{2}{\|y-y_{k}\|}^2} $$ 
    satisfies the following inequality:
   \begin{align}\label{eq_assu1}
  & h(y_{k+1})-l_h(y_{k+1};y_k)
  \leq\frac{\theta_{k+1}}{4}\|y_{k+1}-y_{k}\|^2.
\end{align}
    \State Update $$\lambda_{k+1}\gets\lambda_{k}+\rho \left(F(x_{k+1})+Gy_{k+1}\right).$$ 
    \State $k \gets k+1$
\EndWhile
\end{algorithmic}
\end{algorithm}

\medskip 

\noindent Note that the dominant steps in Algorithm \ref{alg1} are Step 4 and Step 5, as the former involves the nonsmooth function \( g \) in addition to a quadratic term. When \( g \) is convex or weakly convex, the objective function of the subproblem in Step 4 is strongly convex. In contrast, Step 5 involves projecting onto the set \( \mathcal{Y} \), which may be nonconvex. However, when \( \mathcal{Y} \) is convex, the objective function of the subproblem in Step 5 of Algorithm \ref{alg1} is always a strongly convex quadratic function, even if \( h \) is nonconvex.  The dual variables are updated in Step 6 using the conventional update of the dual multipliers in traditional augmented Lagrangian-based methods, see  \cite{RocWet:98}. 
\redd{
\noindent Before proceeding, we introduce an  assumption about the sequence of iterates generated by Algorithm \ref{alg1}, which will play a crucial role in the subsequent convergence analysis.
\begin{assumption} \label{assum:bounded_iter}
The sequence \(\{(x_k, y_k, \lambda_k)\}_{k \geq 0}\) generated by Algorithm \ref{alg1} is bounded.
\end{assumption}}

\medskip 

\noindent \redd{ This assumption is standard in the context of nonconvex nonsmooth optimization, see e.g.,  \cite{HalTeb:23, CohHal:21, BolSab:18}. Boundedness of the primal iterates can be ensured e.g., if $\text{dom }g$ and $\mathcal{Y}$ are bounded sets or the augmented Lagrangian function is coercive or level bounded. However, proving boundedness of the dual iterates in the nonconvex setting remains an open question as pointed out in   \cite{HalTeb:23}.  
}

\medskip 

\noindent Note that  $\beta_k$ and $\theta_k$ in Algorithm \ref{alg1} are well defined since $\psi_{\rho}$  and $h$ are smooth functions according to  Assumption \ref{assump2}. To determine these regularization parameters, one approach is to use a backtracking scheme, as described in Algorithm 2 in \cite{CohHal:21}. 

\medskip

\begin{remark}\label{strange_assum} \redd{If Assumption \ref{assum:bounded_iter} holds and Assumption \ref{assump2} is satisfied on a compact set containing the iterates, then $\psi_{\rho}$ and $h$ are smooth functions on this compact set. Consequently, for any $k \geq 0$, it is always possible to determine $\beta_{k+1}$ and $\theta_{k+1}$ that satisfy \eqref{eq_assu} and \eqref{eq_assu1}, respectively.}
 Moreover, we have:
\begin{equation}\label{bounds_beta_theta}
   \beta:=\sup_{k\geq1}{\beta_{k}}\leq 2L_\psi, \hspace{0.5cm} \theta:=\sup_{k\geq1}{\theta_{k}}\leq 2L_h. 
\end{equation}

\end{remark}

\medskip

\begin{remark}
Note that any descent algorithm, initialized at the current iterate $x_k$, applied for solving the simple (possibly nonconvex) subproblem in Step 4 can always  ensure the  descent:
\begin{align}
\bar{\mathcal{L}}_{\rho}(x_{k+1},y_{k},\lambda_{k};x_{k},y_{k}) & +\frac{\beta_{k+1}}{2}{\|x_{k+1}-x_{k}\|}^2\leq\bar{\mathcal{L}}_{\rho}(x_{k},y_{k},\lambda_{k};x_{k},y_{k}). \label{optimal}
\end{align}
Hence, in the sequel we assume that $x_{k+1}$ automatically satisfies the descent \eqref{optimal}, besides the conditions \eqref{inexactness} and \eqref{eq_assu}. It is also worth mentioning that we can also impose other definitions for inexactness, e.g.,  one can replace the condition \eqref{inexactness}  with the following one: choose  $\beta_{k+1} > \alpha$ such that
\begin{align*}
\bar{\mathcal{L}}_{\rho}(x_{k+1}, y_{k},\lambda_{k}; x_{k}, y_{k}) + \frac{\beta_{k+1}}{2}\|x_{k+1} - x_{k}\|^2 
     \leq \bar{\mathcal{L}}_{\rho}(x_{k}, y_{k}, \lambda_{k}; x_{k}, y_{k}) + \frac{\alpha}{4}\|x_{k+1} - x_{k}\|^2.
\end{align*} 
\end{remark}

\noindent Our convergence results  are also valid under this inexact  setting.   Let us denote  the difference of the steps in $x, y$ and $\lambda$, for all $k\geq 1$ as:
\[
\Delta x_{k}=x_{k}-x_{k-1}, \hspace{0.15cm}\Delta y_{k}=y_{k}-y_{k-1} \hspace{0.15cm} \text{and} \hspace{0.15cm} \Delta\lambda_{k}=\lambda_{k}-\lambda_{k-1}.
\]


\section{Convergence analysis for iL-ADMM}\label{sec4}
\noindent In this section, we first derive the asymptotic convergence, then first-order complexity and finally improved rates under the KL condition for the proposed scheme iL-ADMM (Algorithm \ref{alg1}).  

\subsection{Asymptotic convergence}
\noindent First, let us  derive the asymptotic convergence of iL-ADMM. We start proving the decrease with respect to the first argument of the augmented Lagrangian function. 

\begin{lemma}\label{lemma3}[Descent of $\mathcal{L}_{\rho}$ w.r.t. the first block of primal variables] Let  $\{(x_k,y_k, \lambda_k\}_{k\geq1}$ be the sequence generated by Algorithm \ref{alg1}. If Assumption \ref{assum:bounded_iter} holds and  Assumption \ref{assump2} is satisfied on a compact set where the iterates belong to, then  we have the following descent for $\mathcal{L}_{\rho}$ w.r.t. $x$: 
\begin{align*}
\mathcal{L}_{\rho}(x_{k+1},y_{k},\lambda_{k})\leq\mathcal{L}_{\rho}(x_{k},y_{k},\lambda_{k})-\frac{\beta_{k+1}}{4}\|x_{k+1}-x_{k}\|^2 \quad \forall k \geq 0.
\end{align*}
\end{lemma}
\begin{proof}  From the definition of $x_{k+1}$ and  \eqref{optimal}, we have: 
\begin{align*}
   {\bar{\mathcal{L}}_{\rho}(x_{k+1},y_{k},\lambda_{k};x_{k},y_{k})+\frac{\beta_{k+1}}{2}{\|x_{k+1}-x_{k}\|}^2}\leq\bar{\mathcal{L}}_{\rho}(x_{k},y_{k},\lambda_{k};x_{k},y_{k})={\mathcal{L}}_{\rho}(x_{k},y_{k},\lambda_{k}).
\end{align*}
Further, from definition of $\bar{\mathcal{L}}_{\rho}$ and $\mathcal{L}_{\rho}$, we get: 
\begin{align*}
   & l_f(x_{k+1};x_k)+g(x_{k+1}) +\langle \lambda_{k} ,l_F(x_{k+1};x_{k})\rangle+\frac{\rho}{2}{\|l_F(x_{k+1};x_{k})+Gy_k\|^2}\\
    \leq & f(x_k)+g(x_{k})+\langle \lambda_{k} ,F(x_{k}) \rangle+\frac{\rho}{2}{\|F(x_{k})+Gy_k\|^2}-\frac{\beta_{k+1}}{2}\|\Delta x_{k+1}\|^2.
\end{align*}
Rearranging the above inequality, it follows: 
\begin{align}
  & g(x_{k+1})-g(x_{k})\nonumber\\
   \leq&-\langle \nabla f(x_{k}) ,\Delta x_{k+1} \rangle -\langle \nabla F(x_{k})\Delta x_{k+1},\lambda_{k} \rangle-\frac{\rho}{2}\langle \nabla F(x_{k})\Delta x_{k+1} ,2(F(x_{k})+Gy_k)\rangle\nonumber\\
   &-\frac{\rho}{2}\langle \nabla F(x_{k})\Delta x_{k+1} ,\nabla F(x_{k})\Delta x_{k+1} \rangle-\frac{\beta_{k+1}}{2}\|\Delta x_{k+1}\|^2\nonumber\\
    =&-\langle { \nabla f(x_{k})+\nabla F(x_{k})}^T(\lambda_{k}+\rho(F(x_{k})+Gy_k) ,\Delta x_{k+1} \rangle\nonumber\\
    &-\frac{\rho}{2}\|\nabla F(x_{k})\Delta x_{k+1}\|^2-\frac{\beta_{k+1}}{2}\|\Delta x_{k+1}\|^2\nonumber\\
 \leq&-\langle \nabla_x\psi_{\rho}(x_{k},y_{k},\lambda_{k}) ,\Delta x_{k+1} \rangle -\frac{\beta_{k+1}}{2}\|\Delta x_{k+1}\|^2. \label{use_bellow}
\end{align}
Using the definitions of $\mathcal{L}_{\rho}$ and $\psi_{\rho}$, we further obtain:
\begin{align*}
   & \mathcal{L}_{\rho}(x_{k+1},y_k,\lambda_{k})-\mathcal{L}_{\rho}(x_{k},y_k,\lambda_{k})\\
   &=g(x_{k+1})-g(x_{k})+\psi_{\rho}(x_{k+1},y_k,\lambda_{k})-\psi_{\rho}(x_{k},y_k,\lambda_{k})\\
  &{\overset{{\eqref{eq_assu},\eqref{use_bellow}}}{\leq}}-\frac{\beta_{k+1}}{4}\|x_{k+1}-x_{k}\|^2.   
\end{align*}
This proves our statement.
\end{proof} 

\medskip 
\noindent Let us now prove the decrease with respect to the second argument, $y$, for the augmented Lagrangian function.
\begin{lemma}\label{lemma4}[Descent of $\mathcal{L}_{\rho}$ w.r.t. second block of primal variables]Let  $\{(x_k,y_k, \lambda_k\}_{k\geq1}$ be the sequence generated by Algorithm \ref{alg1}. If Assumption \ref{assum:bounded_iter} holds and  Assumption \ref{assump2} is satisfied on a compact set where the iterates belong to, then  we have the following descent for $\mathcal{L}_{\rho}$ w.r.t. $y$ for all $k\geq0$: 
\[
\mathcal{L}_{\rho}(x_{k+1},y_{k+1},\lambda_{k})\leq\mathcal{L}_{\rho}(x_{k+1},y_{k},\lambda_{k})-\frac{\theta_{k+1}}{4}\|y_{k+1}-y_{k}\|^2.
\]
\end{lemma}
\begin{proof}
Using  the definition of $\mathcal{L}_{\rho}$ and the optimality conditions for $y_{k+1}$, we have: 
\begin{align*}
   {\bar{\mathcal{L}}_{\rho}(x_{k+1},y_{k+1},\lambda_{k};x_{k},y_{k})+\frac{\theta_{k+1}}{2}{\|y_{k+1}-y_{k}\|}^2}&\leq\bar{\mathcal{L}}_{\rho}(x_{k+1},y_{k},\lambda_{k};x_{k},y_{k})\\
   &={\mathcal{L}}_{\rho}(x_{k+1},y_{k},\lambda_{k}).
\end{align*}
Or, we have:
\begin{align*}
  {\mathcal{L}}_{\rho}(x_{k+1},y_{k+1},\lambda_{k})+l_h(y_{k+1};y_k) =\bar{\mathcal{L}}_{\rho}(x_{k+1},y_{k+1},\lambda_{k};x_{k},y_{k})+h(y_{k+1}).
\end{align*}
Then, it follows that:
\begin{align*}
{{\mathcal{L}}_{\rho}(x_{k+1},y_{k+1},\lambda_{k})-{\mathcal{L}}_{\rho}(x_{k+1},y_{k},\lambda_{k})}
& \leq h(y_{k+1})-l_h(y_{k+1};y_k)-\frac{\theta_{k+1}}{2}{\|\Delta y_{k+1}\|}^2\\
& {\overset{{\eqref{eq_assu1}}}{\leq}}-\frac{\theta_{k+1}}{4}{\|\Delta y_{k+1}\|}^2.
\end{align*}
This completes our proof.
\end{proof}

\medskip 

\noindent  Let us now bound the dual variables by the primal variables. 
\begin{lemma}\label{lambda_bound}[Bound for $\|\Delta\lambda_{k+1}\|$]
Let  $\{(x_k,y_k, \lambda_k\}_{k\geq1}$ be the sequence generated by Algorithm \ref{alg1}. If Assumptions \ref{assum_lips} and \ref{assum:bounded_iter} hold and  Assumption \ref{assump2} is satisfied on a compact set where the iterates belong to, then there exists $\kappa > 0$ such that:
\begin{align}
    \label{lambda_squared}
\|\Delta\lambda_{k+1}\|^2\leq 2\frac{\left(\theta_{k+1}+\kappa\right)^2}{\sigma^2}\|\Delta y_{k+1}\|^2+2\frac{(\theta_{k}+L_h)^2}{\sigma^2}\|\Delta y_{k}\|^2.
\end{align}
\end{lemma}
\begin{proof}
 First, using the optimality condition for $y_{k+1}$ combined   with the update in Step 6 of Algorithm \ref{alg1}, we get:
\begin{equation}\label{eq_lambda1}
 {-\nabla h(y_{k})- G^T\lambda_{k+1}-\theta_{k+1} \Delta y_{k+1}\in N_{\mathcal{Y}}(y_{k+1}).}
\end{equation}
By replacing $k$ with $k-1$, we obtain:
\begin{equation}\label{eq_lambda2}
     {-\nabla h(y_{k-1})- G^T\lambda_{k}-\theta_{k}\Delta y_{k}\in N_{\mathcal{Y}}(y_k)}.
\end{equation}
\redd{
Moreover, using Assumptions \ref{assump2} and \ref{assum:bounded_iter}, if follows that $\theta_k \leq \theta$ from \eqref{bounds_beta_theta} and  there exists a constant $0 < R < \infty$ such that:
\begin{equation} \label{radius} 
\left\|-\nabla h(y_{k-1}) - G^T\lambda_{k} - \theta_{k}\Delta y_{k}\right\| \leq R \quad \forall k \geq 1. 
\end{equation}
Since $\mathcal{Y}$ satisfies Assumption \ref{assum_lips},  it follows that there exists $\kappa>0$ such that
\begin{equation}\label{lipschitz_cone}
    \dist_H(\bar{N}_{\mathcal{Y}}(y_k), \bar{N}_{\mathcal{Y}}(y_{k+1})) \leq \kappa \|y_{k+1} - y_k\| \quad \forall k\geq 1,
\end{equation}
where $\bar{N}_{\mathcal{Y}}(y)$ is defined  either in \eqref{normal1} or \eqref{normal2}.
Note that, if $\bar{N}_{\mathcal{Y}}(y)$ is defined by \eqref{normal1}, i.e.:
     \[\bar{N}_{\mathcal{Y}}(y) = N_{\mathcal{Y}}(y) \cap \mathbb{B}_R,\] with $R$ given  in \eqref{radius}, we always have 
     \begin{equation}\label{eq_lambda11}
          {-\nabla h(y_{k-1})- G^T\lambda_{k}-\theta_{k}\Delta y_{k}\in \bar{N}_{\mathcal{Y}}(y_k)} \quad \forall k\geq1.
     \end{equation}
On the other hand, if   $\bar{N}_{\mathcal{Y}}(y)$ is defined by \eqref{normal2}, i.e.:
     \[
      \bar{N}_{\mathcal{Y}}(y) =  \left\{w\in \mathbb{R}^p | w=\nabla H(y)^T\lambda, \; \lambda\in \mathbb{B}_{\frac{R}{\sigma_H}} \right\},
     \] with $R$ given in \eqref{radius} and $\sigma_H>0$ satisfying $\sigma_H \|\lambda\| \leq \|\nabla H(y)^T\lambda\|$ for any $\lambda \in \mathbb{R}^q$ and $y\in \mathcal{Y}$, we have 
     \begin{equation}\label{eq_lambda22}
          {-\nabla h(y_{k-1})- G^T\lambda_{k}-\theta_{k}\Delta y_{k}\in \bar{N}_{\mathcal{Y}}(y_k)} \quad \forall k\geq1.
     \end{equation}
In conclusion, from \eqref{eq_lambda11} and \eqref{eq_lambda22}  and the definition of the Hausdorff distance, we have for any $ k\geq1$: 
\begin{align*}
 \left\|\nabla h(y_{k})-\nabla h(y_{k-1})+G^T\Delta\lambda_{k+1}  + \theta_{k+1}\Delta y_{k+1}-\theta_{k}\Delta y_{k}\right\|
\leq \dist_H(\bar{N}_{\mathcal{Y}}(y_k), \bar{N}_{\mathcal{Y}}(y_{k+1})) .
\end{align*}
Using \eqref{lipschitz_cone} and the triangle inequality, it follows that:
\begin{align*}
 \left\|G^T\Delta\lambda_{k+1}\right\|
\leq\kappa\|\Delta y_{k+1}\|+\|\nabla h(y_{k})-\nabla h(y_{k-1})\|+\theta_{k+1}\|\Delta y_{k+1}\|+\theta_{k}\|\Delta y_{k}\|.
\end{align*}
 Moreover, since $h$ is smooth, we obtain:
\begin{align*}
\left\|G^T\Delta\lambda_{k+1}\right\|
\leq\left(\kappa+\theta_{k+1}\right)\|\Delta y_{k+1}\|+\left(L_h+\theta_{k}\right)\|\Delta y_{k}\|.
\end{align*}
Further, using  Assumption \ref{assump2}, we get $\forall k\geq1$: 
\begin{align}
   \|\Delta\lambda_{k+1}\|
    \leq\frac{1}{\sigma}\Big(\left(\kappa+\theta_{k+1}\right)\|\Delta y_{k+1}\|+(\theta_{k}+L_h)\|\Delta y_{k}\|\Big). \label{delta_to_replace}
\end{align} 
}
Since $(a+b)^2\leq2a^2+2b^2$, we finally get \eqref{lambda_squared}. 
\end{proof}

\medskip 

\noindent  Our convergence proofs use control theoretic tools such as Lyapunov functions. For our algorithm we define the following Lyapunov  function inspired from \cite{XieWri:21} (see also~\cite{CohHal:21}):
\begin{equation}\label{P}
  P(x,y,\lambda,\bar y,\gamma)=\mathcal{L}_{\rho}(x,y,\lambda)+\frac{\gamma}{2}\|y-\bar y\|^2,
\end{equation}
\noindent with  $\gamma>0$ to be defined later. The evaluation of the Lyapunov function along the iterates of Algorithm 1  is denoted by:
\begin{equation}\label{lyapunov_function}
 P_{k}=P(x_k,y_k,\lambda_k,y_{k-1},\gamma_k) \hspace{0.3cm} \forall k\geq1. 
\end{equation}
 In the following lemma, we prove that the Lyapunov function \eqref{P} decreases along the trajectory generated by  iL-ADMM, i.e.,   $\{P_{k}\}_{k\geq1}$ is a decreasing sequence. 
 
\begin{lemma}\label{decrease}[Decrease] Let  $\{(x_k,y_k, \lambda_k\}_{k\geq1}$ be the sequence generated by Algorithm \ref{alg1}. If Assumptions \ref{assum_lips} and \ref{assum:bounded_iter} hold and  Assumption \ref{assump2} is satisfied on a compact set where the iterates belong to.
  Choosing  
 \begin{gather}\label{gamma_rho}
   \red{ \rho\geq 32\frac{(\theta+\max\{L_h,\kappa\})^2}{\theta_0\sigma^2}}, \quad \gamma_{k}= \frac{\theta_{k}}{4},
 \end{gather}
then the Lyapunov function decreases along the iterates according to the following formula:
 \begin{align}
     P_{k+1}-P_{k}\leq&-\frac{\beta_{k+1}}{4}\|\Delta x_{k+1}\|^2-\frac{\theta_{k+1}}{16}\|\Delta y_{k+1}\|^2-\frac{\theta_k}{16}\|\Delta y_{k}\|^2 \quad \forall k\geq1. \label{decrease_Lyapunov}
 \end{align} 
\end{lemma}
\begin{proof} Using the definition of $P_k$ in \eqref{lyapunov_function}, we have
\begin{align}
&P_{k+1}-P_{k}\nonumber\\
=&\mathcal{L}_{\rho}(x_{k+1},y_{k+1},\lambda_{k+1})-\mathcal{L}_{\rho}(x_{k+1},y_{k+1},\lambda_{k})+\mathcal{L}_{\rho}(x_{k+1},y_{k+1},\lambda_{k})-\mathcal{L}_{\rho}(x_{k+1},y_{k},\lambda_{k}) \nonumber\\
&+\mathcal{L}_{\rho}(x_{k+1},y_{k},\lambda_{k})-\mathcal{L}_{\rho}(x_{k},y_k,\lambda_{k})+\frac{\gamma_{k+1}}{2}\|y_{k+1}-y_{k}\|^2-\frac{\gamma_k}{2}\|y_{k}-y_{k-1}\|^2\nonumber\\
{\overset{}{{\leq}}}& \frac{1}{\rho}\|\Delta\lambda_{k+1}\|^2-\frac{\beta_{k+1}}{4}\|\Delta x_{k+1}\|^2-\frac{\theta_{k+1}-2\gamma_{k+1}}{4}\|\Delta y_{k+1}\|^2-\frac{\gamma_k}{2}\|\Delta y_{k}\|^2, \label{lyapunov_to_replace}
\end{align}
where the inequality follows from Lemmas \ref{lemma3}, \ref{lemma4}  and from the update of the dual multipliers in Step 6 of Algorithm \ref{alg1}.
Now, using the inequality \eqref{lambda_squared} in \eqref{lyapunov_to_replace}, we obtain: 
\begin{align*}
   P_{k+1}-P_{k}
 \leq &-\frac{\beta_{k+1}}{4}\|\Delta x_{k+1}\|^2-\frac{\gamma_{k+1}}{4}\|\Delta y_{k+1}\|^2-\frac{\gamma_k}{4}\|\Delta y_{k}\|^2\nonumber\\ 
  &\!-\!\left(\!\frac{\theta_{k+1} \!-\! 3\gamma_{k+1}}{4} -\frac{2\left(\theta_{k+1}\red{+\kappa} \right)^2}{\rho\sigma^2} \! \right) \! \|\Delta y_{k+1}\|^2 \!-\left(\!\frac{\gamma_k}{4}-\frac{2(\theta_k+L_h)^2}{\rho\sigma^2}\!\right)\!\|\Delta y_{k}\|^2\nonumber\\
  {\overset{\eqref{gamma_rho}}{{\leq}}}&-\frac{\beta_{k+1}}{4}\|\Delta x_{k+1}\|^2-\frac{\theta_{k+1}}{16}\|\Delta y_{k+1}\|^2-\frac{\theta_k}{16}\|\Delta y_{k}\|^2\nonumber\\ 
  &-\left(\frac{\theta_{k+1}}{16}-\frac{2(\theta_{k+1}\red{+\kappa})^2}{\rho\sigma^2}\right)\|\Delta y_{k+1}\|^2 -\left(\frac{\theta_k}{16}-\frac{2(\theta_k+L_h)^2}{\rho\sigma^2}\right)\|\Delta y_{k}\|^2.
\end{align*}
 Since
$ \theta_0\leq\theta_k\leq2{L_h}, \; \forall k\geq1$, then by choosing $\rho$ as in \eqref{gamma_rho}, the decrease in \eqref{decrease_Lyapunov} follows.
\end{proof}

\medskip

\noindent  Before proving the global convergence for the iterates generated by  Algorithm \ref{alg1}, let us first  bound  $\partial\mathcal{L}_{\rho}$. Here, $\partial\mathcal{L}_{\rho}$ denotes the limiting subdifferential of $\mathcal{L}_{\rho}$.
 
\begin{lemma}\label{bounded_gradient}[Bound for $\partial\mathcal{L}_{\rho}$]
Let  $\{z_k:=(x_k,y_k, \lambda_k\}_{k\geq1}$ be the sequence generated by Algorithm \ref{alg1}. If Assumption  \ref{assum:bounded_iter} holds and  Assumption \ref{assump2} is satisfied on a compact set where the iterates belong to, then there exists $v_{k+1}\in\partial \mathcal{L}_\rho(z_{k+1})$  satisfying:
\[
    \|v_{k+1}\|\leq c\|z_{k+1}-z_{k}\| \quad k\geq 1,
\]
where $c = L_{\psi}+\beta+\alpha+\rho M_F^2 + L_h + \theta + \|G\| + \rho^{-1}$.
\end{lemma}
\begin{proof}
 We note that for every $v_{k+1} = (v^x_{k+1}, v^y_{k+1}, v^{\lambda}_{k+1}) \in \partial \mathcal{L}_{\rho}(z_{k+1})$, we have
\begin{align}
    v^x_{k+1} &\in \partial_x\mathcal{L}_{\rho}(x_{k+1}, y_{k+1}, \lambda_{k+1})= \partial g(x_{k+1}) + \nabla_x\psi_{\rho}(x_{k+1}, y_{k+1}, \lambda_{k+1}), \label{eq:4.30} \\
    v^y_{k+1} &\red{\in} \partial_y\mathcal{L}_{\rho}(x_{k+1}, y_{k+1}, \lambda_{k+1}) \nonumber \\
    & = \nabla h(y_{k+1}) + \nabla_y\psi_{\rho}(x_{k+1}, v_{k+1}, \lambda_{k+1}) +\red{N_{\mathcal{Y}}(y_{k+1})}, \label{eq:4.31}\\
    v^{\lambda}_{k+1} &= \nabla_{\lambda}\mathcal{L}_{\rho}(x_{k+1}, y_{k+1}, \lambda_{k+1})= F(x_{k+1}) - Gy_{k+1}. \nonumber
\end{align}
Using the optimality \eqref{inexactness}, it follows that there exists $s_{k+1}^g \in \partial g(x_{k+1})$ such that
\begin{align*}
    \left\|s_{k+1}^g +\nabla_x\psi_{\rho}(x_k, y_k, \lambda_k) + \beta_{k+1}\Delta x_{k+1} +\rho\nabla F(x_k)^T\nabla F(x_k)\Delta x_{k+1}\right\|\leq\alpha\|\Delta x_{k+1}\|.
\end{align*}
It follows that there exists $v^x_{k+1}$ as in \eqref{eq:4.30}, such that
\begin{align*}
    \|v^x_{k+1}\| \leq \|\nabla_x\psi_{\rho}(x_{k+1}, y_{k+1}, \lambda_{k+1})-\nabla_x\psi_{\rho}(x_k, y_k, \lambda_k)\| + (\beta_{k+1}+\alpha+\rho M_F^2)\|\Delta x_{k+1}\|.
\end{align*}
Since $\nabla_x\psi_{\rho}$ is locally Lipschitz continuous, noting that the sequence $\{ z_k := (x_k, y_k, \lambda_k) \}_{k\geq1}$ is bounded and $\beta_k \leq \beta$, it follows that:
\begin{align}
    \|v^x_{k+1}\| &\overset{\eqref{bounds_beta_theta}}{\leq} (L_{\psi}+\beta+\alpha+\rho M_F^2)\|z_{k+1}-z_k\|. \label{eq:bound_x}
\end{align}
Next, we note that
\begin{equation}
    \nabla_y\psi_{\rho}(x_{k+1}, y_{k+1}, \lambda_{k+1}) = G^T(2\lambda_{k+1} - \lambda_k), \label{eq:4.34}
\end{equation}
where the second equality is due to \redd{the} multiplier update \redd{at} Step 6 of Algorithm \ref{alg1}. Moreover, from the first\redd{-order} optimality condition of Step 5 of A\redd{l}gorithm  \ref{alg1}, we have:
\red{\[
-\nabla h(y_k)-\theta_{k+1}\Delta y_{k+1}\in G^T\lambda_{k+1}+N_{\mathcal{Y}}(y_{k+1}).
\]
It then follows that:
\begin{align*}
   G^T\Delta\lambda_{k+1} - \nabla h(y_k) - \theta_{k+1}\Delta y_{k+1}\in \nabla_y\psi_{\rho}(x_{k+1}, y_{k+1}, \lambda_{k+1})+N_{\mathcal{Y}}(y_{k+1}). 
\end{align*}}
Together with \eqref{eq:4.31}, we have:
\begin{align}
    \|v^y_{k+1}\| 
    &\leq   \|\nabla h(y_{k+1}) - \nabla h(y_k)\| +\|G\|\|\Delta\lambda_{k+1}\| + \theta_{k+1} \|\Delta y_{k+1}\| \nonumber \\
    &\leq (L_h + \theta_{k+1})\|\Delta y_{k+1}\| + \|G\|\|\Delta\lambda_{k+1}\| \overset{\eqref{bounds_beta_theta}}{\leq} (L_h + \theta + \|G\|)\|z_{k+1} - z_k\|. \label{eq:bound_v}
\end{align}
Finally, we note that
\begin{equation}
    \|v^{\lambda}_{k+1}\| = \left\| F(x_{k+1}) + Gy_{k+1} \right\| = \rho^{-1}\|\Delta \lambda_{k+1}\|,\label{eq:4.37}
\end{equation}
where the second equality in \eqref{eq:4.37} is due to \redd{the} multiplier update \redd{at} Step 6 in Algorithm \ref{alg1}. Thus, by summing the bounds for $\|v^x_{k+1}\|$, $\|v^y_{k+1}\|$, and $\|v^{\lambda}_{k+1}\|$, we get:
\begin{equation}
    \|v_{k+1}\| \leq c \|z_{k+1} - z_k\|,
\end{equation}
with $c = L_{\psi}+\beta+\alpha+\rho M_F^2 + L_h + \theta + \|G\| + \rho^{-1} > 0$.
\end{proof}

\medskip 

\noindent Let us now present  the global asymptotic convergence for the iterates of   Algorithm \ref{alg1}. 
\begin{theorem}\label{unused_lemma}[Limit points are stationary points]Let  $\{z_k:=(x_k,y_k, \lambda_k\}_{k\geq1}$ be the sequence generated by Algorithm \ref{alg1}. If Assumptions \ref{assum_lips} and \ref{assum:bounded_iter} hold,  Assumption \ref{assump2} is satisfied on a compact set where the iterates belong to and  $\rho$ is chosen as in Lemma \ref{decrease}, then any limit point $z^*:=(x^*,y^*,\lambda^*)$ of $\{z_k\}_{k\geq1}$ is a stationary point of the augmented Lagrangian function, i.e., $0\in\partial\mathcal{L}_{\rho}(x^*,y^*,\lambda^*)$. Equivalently, $z^*$ is a KKT point of problem \eqref{eq1}.
\end{theorem}
\begin{proof}
Since $\beta\geq\beta_0$ and $\theta_k\geq\theta_0$, for any $ k\geq1$, it then follows  from \eqref{decrease_Lyapunov}  that, for any $ k\geq1$ we have: 
\begin{align*}
     \frac{\beta_0}{4}\|&\Delta x_{k+1}\|^2+\frac{\theta_0}{16}\|\Delta y_{k+1}\|^2+\frac{\theta_0}{16}\|\Delta y_{k}\|^2\leq P_{k}-P_{k+1}.
\end{align*}
Let $k\geq1$, by summing up the above inequality from $i=1$ to $i=k$, we obtain:
\begin{align}
\sum_{i=1}^{k}\left(\frac{\beta_0}{4}\|\Delta x_{\redd{i}+1}\|^2+\frac{\theta_0}{16}\|\Delta y_{\redd{i}+1}\|^2+\frac{\theta_0}{16}\|\Delta y_{\redd{i}}\|^2\right)\leq P_{1}-P_{k+1}
{\leq}P_1-\bar{P},\label{limit}
\end{align}
where $\bar{P}$ is a lower bound on the sequence $\{P_k\}_{k\geq1}$
(it is finite since $z_k$ is bounded).  Since \eqref{limit} holds for any $k\geq1$, we have:
\[
\sum_{i=1}^{\infty}{\left( \frac{\beta_0}{4}\|\Delta x_{k+1}\|^2+\frac{\theta_0}{16}\|\Delta y_{k+1}\|^2+\frac{\theta_0}{16}\|\Delta y_{k}\|^2\right)}<\infty.
\]
This, together with the fact that $\beta_0,\theta_0>0$, yields that:
\begin{equation}\label{zero_limit}
    \lim_{k\to\infty}{\|\Delta x_{k}\|}=0 \hspace{0.3cm}\text{ and }  \lim_{k\to\infty}{\|\Delta y_{k}\|}=0.
\end{equation}
From Lemma \ref{lambda_bound}, it follows that \[
  \lim_{k\to\infty}{\|z_{k+1}-z_k\|}=0.
\]
Since the sequence $\{(x_{k}, y_k, \lambda_k)\}_{k\geq1}$ is bounded, according to Assumption \ref{assum:bounded_iter},  there exists a convergent subsequence, let us say  $\{(x_{k},y_k,\lambda_{k})\}_{k\in\mathcal{K}}$, with the limit $(x^*,y^*,\lambda^*)$.
From {Lemma \ref{bounded_gradient}}, we have $v_{k+1}\in\partial\mathcal{L}_{\rho}(z_{k+1})$ such that:
\[
\lim_{k\in\mathcal{K}}{\|v_{k+1}\|}\leq c\lim_{k\in\mathcal{K}}\|z_{k+1}-z_{k}\|=0.
\]
Thus, from the closedness of the map $\partial\mathcal{L}_{\rho}$,   it follows that $0\in\partial\mathcal{L}_{\rho}(x^*,y^*,\lambda^*)$, which completes the proof.
\end{proof}


\subsection{First-order convergence rate}
\noindent Let us now derive the complexity (i.e., convergence rate) of the proposed method for finding an $\epsilon$-KKT point of problem \eqref{eq1}.
\noindent For the remainder of this paper, we define:
\begin{equation}\label{bar_gamma}
\ubar{\gamma}:=\min\{4\beta_0,\theta_0\}>0.   
\end{equation}

\begin{theorem}\label{main_result}[First-order complexity]
Let  $\{z_k:=(x_k,y_k, \lambda_k\}_{k\geq1}$ be the sequence generated by Algorithm \ref{alg1}. If Assumptions \ref{assum_lips} and \ref{assum:bounded_iter} hold,  Assumption \ref{assump2} is satisfied on a compact set where the iterates belong to, and  $\rho$ is chosen as in Lemma \ref{decrease},  then for any $\epsilon>0$, Algorithm \ref{alg1} yields an $\epsilon$-first-order solution of \eqref{eq1} after $K= 16c^2\left({1+2\frac{\theta+\max\{L_h,\kappa\}}{\sigma}}\right)^2\left(\frac{P_1-\bar{P}}{\ubar{\gamma} }\right)\frac{1}{\epsilon^2}$ iterations.
\end{theorem}
\begin{proof}
Let $K\geq1$, then from \eqref{limit} and \eqref{bar_gamma}, we have:
\begin{align*}
    \frac{\ubar{\gamma}}{16}\sum_{i=1}^{K}{\left(\|\Delta x_{i+1}\|^2+\|\Delta y_{i+1}\|^2+\|\Delta y_i\|^2\right)}\leq P_1-\bar{P},
\end{align*}
we recall that $\bar{P}$ is a lower bound on the sequence $\{P_k\}_{k\geq1}$. 
Therefore, there exists $k^*\in\{1,...,K\}$ such that:
\[
{\|\Delta x_{k^*+1}\|^2+\|\Delta y_{k^*+1}\|^2+\|\Delta y_{k^*}\|^2}\leq 16\frac{P_1-\bar{P}}{K{\ubar{\gamma}}}.
\]
It implies that: $ \|\Delta x_{k^*+1}\|\leq4\sqrt{\frac{(P_1-\bar{P})}{K\ubar{\gamma}}}, $ $\|\Delta y_{k^*+1}\|\leq4\sqrt{\frac{(P_1-\bar{P})}{K\ubar{\gamma}}} $ and
$\|\Delta y_{k^*}\|\leq4\sqrt{\frac{(P_1-\bar{P})}{K\ubar{\gamma}}}$. Hence, from Lemma \ref{bounded_gradient} and \eqref{delta_to_replace}, there exists $v_{k^*+1}\in\partial\mathcal{L}_{\rho}(x_{k^*+1},y_{k^*+1},\lambda_{k^*+1})$ such that:
\begin{align*}
        \|v_{k^*+1}\|\leq c \|\Delta z_{k^*+1}\|\leq 4c\left({1+2\frac{\theta+\red{\max\{L_h,\kappa\}}}{\sigma}}\right)\sqrt{\frac{(P_1-\bar{P})}{K\ubar{\gamma}}}.
\end{align*}
It  follows that for any $ \epsilon>0 $,  $ \|v_{k^*+1}\|\leq\epsilon$  when  $ K\geq 16c^2\left({1+2\frac{\theta+\max\{L_h,\kappa\}}{\sigma}}\right)^2\left(\frac{P_1-\bar{P}}{\ubar{\gamma} \epsilon^2}\right)$.  
Consequently, after \[K= 16c^2\left({1+2\frac{\theta+\max\{L_h,\kappa\}}{\sigma}}\right)^2\left(\frac{P_1-\bar{P}}{\ubar{\gamma} }\right)\frac{1}{\epsilon^2}\]  iterations, Algorithm \ref{alg1} yields  an $\epsilon$-first-order  solution of problem \eqref{eq1}. This concludes our proof. 
\end{proof}

\noindent From Theorem \ref{main_result}, it follows that Algorithm \ref{alg1} yields an $\epsilon$-KKT point of problem \eqref{eq1} within $\mathcal{O}\left(\frac{1}{\epsilon^2}\right)$ iterations, hence matching the optimal complexity of first-order methods for solving nonconvex nonsmooth problems, see  e.g., \cite{Ber:96, BirMar:14, CohHal:21, BolSab:18, BotNgu:20, ThemPat:20}.


\subsection{Improved convergence rate under KL}
\noindent  Previous theorems shows  that (limit) points of the sequence $\{z_{k}:=(x_{k},y_k,\lambda_{k})\}_{k\geq1}$ generated by Algorithm \ref{alg1} are $\epsilon$-KKT  (stationary points) of problem \eqref{eq1}, respectively. The goal in this section is to prove that under the additional KL condition (see Definition \ref{def2})  the whole sequence  $(z_k)_{k\geq 1}$ generated by Algorithm \ref{alg1} converges to a KKT point of problem \eqref{eq1} and derive also improved rates. Recall that  the KL property holds for a large class of functions including semi-algebraic functions  and sum of square functions with uniform non-degenerate Jacobian \cite{BolDan:07}.   In order to show these results,  we first bound the full gradient $\nabla P(\cdot)$ (recall that  $P(\cdot)$ is the Lyapunov function defined in \eqref{P}).
\redd{Throughout this section, for simplicity of the exposition, we assume that $g$ is a continuous function.}
\begin{lemma}\label{bounded_grad}[Boundedness of $\nabla P$]  Let  $\{z_k:=(x_k,y_k, \lambda_k\}_{k\geq1}$ be the sequence generated by Algorithm \ref{alg1}. If Assumption  \ref{assum:bounded_iter} holds and  Assumption \ref{assump2} is satisfied on a compact set where the iterates belong to and  $\{P_{k}\}_{k\geq1}$ is defined as in \eqref{lyapunov_function}, then there exists $p_{k+1}\in\partial P(x_{k+1},y_{k+1},\lambda_{k+1},y_{k},\gamma_{k+1})$, such that for any $k\geq0$, we have:
\begin{align*}
   \|p_{k+1}\|\leq \left(\frac{2c+\theta+D_S}{2}\right)\|z_{k+1}-z_{k}\|,
\end{align*}
where $c$ is defined in Lemma \ref{bounded_gradient}.
\end{lemma}
\begin{proof}
If $p_{k+1}\in\partial  P(x_{k+1},y_{k+1},\lambda_{k+1},y_{k},\gamma_{k+1})$, then there exists $v_{k+1}\in\partial \mathcal{L}_\rho(z_{k+1})$ such that: 
\[
\|p_{k+1}\|\leq\|v_{k+1}\|+2\gamma_{k+1}\|\Delta y_{k+1}\|+\frac{1}{2}\|\Delta y_{k+1}\|^2.
\]
By defining $D_S:=\sup_{y,y'\in\mathcal{S}_y}\|y-y'\|$ and making use of Lemma \ref{bounded_gradient}, it follows that:
\begin{align*}
    \|p_{k+1}\|&\leq\|v_{k+1}\|+2\gamma_{k+1}\|\Delta y_{k+1}\|+\frac{D_S}{2}\|\Delta y_{k+1}\|\\
    &\overset{\eqref{gamma_rho},\eqref{bounds_beta_theta}}{\leq }\left(c+\frac{\theta+D_S}{2}\right)\|z_{k+1}-z_k\|.
\end{align*}
\end{proof}
\noindent The above lemma directly implies the following:
\begin{align}
   \|p_{k+1}\|^2\leq\left(\frac{2c+\theta+D_S}{2}\right)^2\|\Delta z_{k+1}\|^2. \label{key_formul}
\end{align}

Then, it follows from \eqref{key_formul} and \eqref{decrease_Lyapunov}, that:
 \begin{equation}\label{rate}
        P_{k+1}-P_{k}\leq-\frac{\ubar{\gamma}}{4(2c+\theta+D_S)^2}\left\|p_{k+1}\right\|^2.
 \end{equation}
\noindent Let us denote $z_{k}=(x_{k},y_k,\lambda_{k})$ and  $u_{k}=(x_{k},y_k,\lambda_{k},x_{k-1},\gamma_k)$. Moreover,  $\crit P$ denotes the set of critical points of the Lyapunov function $P$ defined in \eqref{lyapunov_function}. Furthermore, we denote $\mathcal{E}_{k}=P_{k}-P^{*}$, where $P^{*}=\lim_{k\to\infty}{P_{k}}$ (recall that  the sequence $\{P_k\}_{k\geq1}$ is decreasing and bounded from below, hence it is convergent). Let us denote the set of limit points of $\{u_k\}_{k\geq1}$ by:
\begin{align*}
    \Omega:=\{u^{*}\;:\; \exists \text{ a convergent subsequence} \; \{u_k\}_{k\in\mathcal{K}} \; \text{such that} \lim_{k\in\mathcal{K}}{u_k}=u^{*}\}.
\end{align*}
Let us now prove the following lemma.
\begin{lemma}\label{added_lemma}
  Let  $\{z_k:=(x_k,y_k, \lambda_k)\}_{k\geq1}$ be the sequence generated by Algorithm \ref{alg1}. If Assumption  \ref{assum:bounded_iter} holds,  Assumption \ref{assump2} is satisfied on a compact set where the iterates belong to,  $\{P_{k}\}_{k\geq1}$ is defined as in \eqref{lyapunov_function} \redd{and the function $g$ is continuous}, then the following  hold:
\begin{enumerate}[(i)]
  \item $\Omega$  is a compact subset of $\redd{\crit P}$ and   $ \lim_{k\to\infty}{\redd{\dist}(u_k,\Omega)}=0$.\label{lem_item1}
     \item For any $u\in\Omega,$ we have $P(u)=P^{*}$.\label{lem_item2}
  \item  For any $(x,y,\lambda, z, \gamma)\in\crit P,$ we have that $(x,y,\lambda)$ is a stationary point of original problem  \eqref{eq1}. \label{lem_item3}
\end{enumerate}
\end{lemma}
\begin{proof}
    \eqref{lem_item1} Since $\{u_k\}_{k\geq1}$ is bounded, there exists   a convergent subsequence  $\{u_k\}_{k\in\mathcal{K}}$ such that $\lim_{k\in\mathcal{K}}{u_k}=u^{*}$. Hence $\Omega$ is nonempty. Moreover, $\Omega$ is compact since it is bounded and closed.
On the other hand, for any $u^{*}\in\Omega$, there exists a sequence of increasing integers $\mathcal{K}$ such that $\lim_{k\in\mathcal{K}}{u_k}=u^{*}$ and using Lemma \ref{bounded_grad} and \eqref{zero_limit}, it follows that there exists $p^*\in\partial  P(u^*)$:
\[
\|p^*\|=\lim_{k\in\mathcal{K}}{\|p_{k+1}\|}=0.
\]
Hence, $u^{*}\in\crit P$ and $0\leq \lim_{k\to\infty}{\dist(u_k,\Omega)}\leq\lim_{k\in\mathcal{K}}{\dist(u_k,\Omega)}=\dist(u^{*},\Omega)=0$.

\noindent \eqref{lem_item2} Since $P$ is continuous (recall that problem's data are all assumed to be continuous functions) and $ \{P(u_k)=P_k\}_{k\geq1}$ converges to $P^{*}$, then any subsequence $\{P(u_k)=P_k\}_{k\in\mathcal{K}}$ that converges, it must converge to the same limit $P^{*}$.

\noindent \eqref{lem_item3} Let $(x,y,\lambda,z,\gamma)\in\crit P$, that is,  there exists  $0\in\partial P(x,y,\lambda,\bar y,\gamma)$. It then follows:
\begin{gather*}
    0\in\partial_x P(x,y,\lambda,\bar y,\gamma)=\partial_x \mathcal{L}_{\rho}(x,y,\lambda), \\
    \red{0\in\partial_{y}{P}(x,y,\lambda,\bar y,\gamma)=\partial_{y}{\mathcal{L}_{\rho}}(x,y,\lambda)+\gamma(y-\bar y)}\\
    \nabla_{\lambda}{P}(x,y,\lambda,\bar y,\gamma)=\nabla_{\lambda}{\mathcal{L}_{\rho}}(x,y,\lambda)=0\\
      \nabla_{\bar y}{P}(x,y,\lambda,\bar y,\gamma)={{\gamma}}(\bar y-y) =0\\
      \nabla_{\gamma}{P}(x,y,\lambda,\bar y,\gamma)=\frac{1}{2}\|y-\bar y\|^2=0.
\end{gather*}
With some minor rearrangements, we obtain:
\begin{gather*}
    -\nabla f(x)-{\nabla F(x)}^T\lambda\in\partial g(x), \quad \red{0\in\nabla h(y)+G^T\lambda+N_{\mathcal{Y}}(y)}, \\
   F(x)+Gy=0.
\end{gather*}
Hence, $(x,y,\lambda)$ is a stationary point of \eqref{eq1}.
\end{proof}

\medskip  

\noindent  Let us now prove that the sequence $\Big\{{\|\Delta x_{k}\|+\|\Delta y_k\|+\|\Delta\lambda_{k}\|}\Big\}_{k\geq1}$ has  finite length, provided that $P$ satisfies the KL property.  \redd{It is known that e.g., semi-algebraic functions satisfy the KL condition and they are stable under operations such as addition and multiplication. Therefore, if the data functions in our problem (i.e., $f$, $g$, $h$, and $F$) are e.g., semi-algebraic, the Lyapunov function $P$ defined  in \eqref{P} will be also semi-algebraic,  thus satisfying  the KL property.}

\begin{lemma}\label{finite_length} Let  $\{z_k:=(x_k,y_k, \lambda_k)\}_{k\geq1}$ be the sequence generated by Algorithm \ref{alg1}. If Assumption  \ref{assum:bounded_iter} holds,  Assumption \ref{assump2} is satisfied on a compact set where the iterates belong to,  $\{P_{k}\}_{k\geq1}$ is defined as in \eqref{lyapunov_function}, \redd{ the function $g$ is continuous}, $\rho$ is chosen as in Lemma \ref{decrease} and, additionally,   assume that $P$ defined in \eqref{P} satisfies the KL property on $\Omega$, then  $\{z_{k}\}_{k\geq1}$  satisfies the finite length property, i.e.:
\[
\sum_{k=1}^{\infty}{\|\Delta x_{k}\|+\|\Delta y_{k}\|+\|\Delta\lambda_{k}\|}<\infty,
\]
and consequently converges to a stationary point of \eqref{alg1}.
\end{lemma}
\begin{proof}
From  boundedness of $\|\Delta \lambda_{k+1}\|^2$ (see \eqref{lambda_squared}), we have the following:
 \begin{align}\label{llambda}
  \|\Delta\lambda_{k+1}\|^2&\leq 2\frac{\theta_{k+1}^2}{\sigma^2}\|\Delta y_{k+1}\|^2+2\frac{(\theta_{k}+L_h)^2}{\sigma^2}\|\Delta y_{k}\|^2\nonumber\\
    &\leq2\frac{(\theta+L_h)^2}{\sigma^2}\left(\|\Delta y_{k+1}\|^2+\|\Delta y_{k}\|^2\right).
 \end{align}
Adding the term $\|\Delta x_{k+1}\|^2+\|\Delta y_{k+1}\|^2+\|\Delta y_{k}\|^2$ on both sides in \eqref{llambda}, we obtain:
\begin{align}\label{z_k}
 \|z_{k+1}-z_{k}\|^2&=\|\Delta x_{k+1}\|^2+\|\Delta y_{k+1}\|^2+\|\Delta\lambda_{k+1}\|^2\nonumber\\
 &\leq\|\Delta x_{k+1}\|^2+\|\Delta y_{k+1}\|^2+\|\Delta\lambda_{k+1}\|^2+\|\Delta y_{k}\|^2\nonumber\\
 &{\overset{{\eqref{llambda}}}{\leq}}\left(2\frac{(\theta+\red{\max\{L_h,\kappa\}})^2}{\sigma^2}+1\right)\left(\|\Delta x_{k+1}\|^2+\|\Delta y_{k+1}\|^2+\|\Delta y_{k}\|^2\right). 
 \end{align}
Considering \eqref{bar_gamma}, we can then rewrite \eqref{decrease_Lyapunov} as follows: 
\begin{align}\label{llyap}
     P_{k+1}-P_{k}&{\overset{{\eqref{decrease_Lyapunov}}}{\leq}}-\frac{\ubar{\gamma}}{16}\left(\|\Delta x_{k+1}\|^2+\|\Delta y_{k+1}\|^2+\|\Delta y_{k}\|^2\right)\nonumber\\
    &{\overset{{\eqref{z_k}}}{\leq}}-\frac{\ubar{\gamma}}{16\left(2\frac{(\theta+\red{\max\{L_h,\kappa\}})^2}{\sigma^2}+1\right)}\|z_{k+1}-z_{k}\|^2.
\end{align}
 Since $ P_{k}\to P^{*}$ and  $\{P_{k}\}_{k\geq 1}$ is monotonically decreasing to $P^{*}$, then it follows that the error sequence $\{\mathcal{E}_{k}\}_{k\geq 1}$ is non-negative, monotonically decreasing and converges to $0$. We distinguish  two cases.

\medskip 

\noindent \textbf{{Case 1}}: There exists  $k_1\geq 1$ such that $\mathcal{E}_{k_1}=0$. Then, $\mathcal{E}_{k}=0 \;  \forall k\geq k_1$ and using \eqref{llyap}, we have:
\[
\|z_{k+1}-z_{k}\|^2\leq\frac{16\left(2\frac{(\theta+\red{\max\{L_h,\kappa\}})^2}{\sigma^2}+1\right)}{\ubar{\gamma}}(\mathcal{E}_{k}-\mathcal{E}_{k+1})=0 \hspace{0.2cm}\forall k\geq k_1.
\]
Since the sequence $\{z_{k}\}_{k\geq1}$ is bounded, we have:
 \begin{align*}
 \sum_{k=1}^{\infty}{\|\Delta x_{k}\|+\|\Delta y_{k}\|+\|\Delta\lambda_{k}\|}=\sum_{k=1}^{k_1}{\|\Delta x_{k}\|+\|\Delta y_{k}\|+\|\Delta\lambda_{k}\|}{\overset{{}}{<}}\infty.
 \end{align*}
 
 \noindent \textbf{{Case 2}}: The error $\mathcal{E}_{k}>0 \;  \forall k\geq 1$. Then,  there exists  $k_1=k_1(\epsilon,\tau)\geq 1$  such that $\forall k\geq k_1$ we have $\dist(u_k,\Omega)\leq \epsilon$,  $P^{*}<P(u_k)<P^{*}+\tau$
 and
 \begin{equation}\label{KL}
     \varphi'(\mathcal{E}_{k})\|\partial P(x_{k},y_k,\lambda_{k},y_{k-1},\gamma_k)\|\geq1,
 \end{equation}
where $ \epsilon>0, \tau>0$ and $\varphi\in\Psi_{\tau}$ are  defined from the KL property of  $P$  on $\Omega$. Since $\varphi$ is concave, we have $\varphi(\mathcal{E}_{k})-\varphi(\mathcal{E}_{k+1})\geq\varphi'(\mathcal{E}_{k})(\mathcal{E}_{k}-\mathcal{E}_{k+1})$. Then, from \eqref{llyap} and \eqref{KL}:
 \begin{align*}
 \|z_{k+1}-z_{k}\|^2
 &{\overset{\eqref{KL}}{\leq}}\varphi'(\mathcal{E}_{k})\|z_{k+1}-z_{k}\|^2\|\partial P(x_{k},y_k,\lambda_{k},z_{k-1},\gamma_k)\|\nonumber\\
 &{\overset{\eqref{llyap}}{\leq}}\frac{16\left(2\frac{(\theta+\red{\max\{L_h,\kappa\}})^2}{\sigma^2}+1\right)}{\ubar{\gamma}}\varphi'(\mathcal{E}_{k})(\mathcal{E}_{k}-\mathcal{E}_{k+1})\|\partial P(x_{k},y_k,\lambda_{k},z_{k-1},\gamma_k)\|\nonumber\\
 &\leq\frac{16\left(2\frac{(\theta+\red{\max\{L_h,\kappa\}})^2}{\sigma^2}+1\right)}{\ubar{\gamma}}\Big(\varphi(\mathcal{E}_{k})-\varphi(\mathcal{E}_{k+1})\Big)\|\partial P(x_{k},y_k,\lambda_{k},z_{k-1},\gamma_k)\|.
 \end{align*}
Since $ \|\Delta z_{k+1}\|^2={\|\Delta x_{k+1}\|^2+\|\Delta y_{k+1}\|^2+\|\Delta\lambda_{k+1}\|^2}$. Using the fact that for any $a,b,c,d,e\geq0$, if $ {a^2+b^2+c^2}\leq d\times e$, then $ (a+b+c)^2\leq 4(a^2+b^2+c^2)\leq 4d\times e\leq 2(d^2+e^2)\leq 4(d+e)^2$, it follows that for any $\eta>0$, we have:
\begin{align}\label{lmit}
 &\|\Delta x_{k+1}\|+\|\Delta y_{k+1}\|+\|\Delta\lambda_{k+1}\|\nonumber\\
    \leq& \frac{32\left(2\frac{(\theta+\red{\max\{L_h,\kappa\}})^2}{\sigma^2}+1\right)\eta}{\ubar{\gamma}}\Big(\varphi(\mathcal{E}_{k})-\varphi(\mathcal{E}_{k+1})\Big) +\frac{2}{\eta}\|\partial P(x_{k},\lambda_{k},x_{k-1},\gamma_k)\|.
\end{align}
Furthermore, from Lemma \ref{bounded_grad}, there exists $p_{k}\in\partial P(x_{k},y_{k},\lambda_{k},y_{k-1},\gamma_{k})$ such that:
\begin{align*}
   \|p_{k}\|\leq \left(\frac{2c+\theta+D_S}{2}\right)\|z_{k}-z_{k-1}\|.
\end{align*}
It then follows that:
\begin{align}
    \|\Delta x_{k+1}\|+\|\Delta y_{k+1}\|+\|\Delta\lambda_{k+1}\|
    \leq&\frac{32\left(2\frac{(\theta+\red{\max\{L_h,\kappa\}})^2}{\sigma^2}+1\right)\eta}{\ubar{\gamma}}\Big(\varphi(\mathcal{E}_{k})-\varphi(\mathcal{E}_{k+1})\Big)\nonumber\\
    &+\frac{2c+\theta+D_S}{\eta}\left(\|\Delta x_{k}\|+\|\Delta y_{k}\|+\|\Delta\lambda_{k}\|\right). \label{used_later}
\end{align}
Let us now choose $\eta>0$ such that $0<\frac{2c+\theta+D_S}{\eta}<1$ and define a parameter $\delta_0$ as  $\delta_0=1-\frac{2c+\theta+D_S}{\eta}>0$. Then, by
summing up the above inequality from $k=\ubar{k}\geq k_1$ to $k=K$ and using the property: $\sum_{k=\ubar{k}}^{K}{\|\Delta z_{k}\|}=\sum_{k=\ubar{k}}^{K}{\|\Delta z_{k+1}\|}+\|\Delta z_{\ubar{k}}\|-\|\Delta z_{{K+1}}\|$, we get:  
\begin{align*}
   \sum_{k=\ubar{k}}^{K}{\|\Delta x_{k+1}\|+\|\Delta y_{k+1}\|+\|\Delta\lambda_{k+1}\|}
  \leq&\frac{32\left(2\frac{(\theta+\red{\max\{L_h,\kappa\}})^2}{\sigma^2}+1\right)\eta}{\ubar{\gamma}\delta_0}\varphi(\mathcal{E}_{\ubar{k}})\nonumber\\
    &+\frac{2c+\theta+D_S}{\eta\delta_0}\Big(\|\Delta x_{\ubar{k}}\|+\|\Delta y_{\ubar{k}}\|+\|\Delta\lambda_{\ubar{k}}\|\Big).
\end{align*}
It is clear that the right-hand side of the above inequality is bounded for any $K\geq\ubar{k}$. Letting  $K\to\infty$, we get that:
\[
    \sum_{k=\ubar{k}}^{\infty}{\|\Delta x_{k+1}\|+\|\Delta y_{k+1}\|+\|\Delta\lambda_{k+1}\|}<\infty.
\]
Since the sequence  $\{(x_{k},y_k,\lambda_{k})\}_{k\geq1}$ is bounded, it follows that:
\[
    \sum_{k=1}^{\ubar{k}}{\|\Delta x_{k}\|+\|\Delta y_{k}\|+\|\Delta\lambda_{k}\|}<\infty.
\]
Hence: $ \sum_{k=1}^{\infty}{\|\Delta x_{k}\|+\|\Delta y_{k}\|+\|\Delta\lambda_{k}\|}<\infty$. 
Let $m, n\in\redd{\mathbb{Z}}_{+}$ such that $n\geq m$, we have:
\begin{align*}
    \|z_n-z_m\|=\|\sum_{k=m}^{n-1}{\Delta z_{k+1}}\|
    \leq\sum_{k=m}^{n-1}{\|\Delta z_{k+1}\|}\leq\sum_{k=m}^{n-1}{\|\Delta x_{k+1}\|+\|\Delta y_{k+1}\|+\|\Delta\lambda_{k+1}\|}.
\end{align*}
Since  $ \sum_{k=\redd{1}}^{\infty}{\|\Delta x_{k+1}\|+\|\Delta y_{k+1}\|+\|\Delta\lambda_{k+1}\|}<\infty$, it follows that $\forall \varepsilon>0, \exists N\in\redd{\mathbb{Z}}_{+}$ such that $\forall m, n\geq N$ where $n\geq m$, we have: $ \|z_n-z_m\|\leq\varepsilon$. This implies that $\{z_k\}_{k\geq1}$ is a Cauchy sequence and thus converges. Moreover, by Theorem \ref{unused_lemma}, $\{z_k\}_{k\geq1}$  converges to a stationary point of \eqref{eq1}.
This concludes our proof.
\end{proof}

 \medskip 
 
\noindent  Lemma \ref{finite_length} shows  that the set of  limit points of the sequence $\{(x_{k},y_k,\lambda_{k})\}_{k\geq1}$ is a singleton. Let us  denote its limit by $(x^{*},y^*,\lambda^{*})$. We are now ready to present the convergence rates of the whole sequence generated by  Algorithm \ref{alg1} (see also  \cite{Yas:22} for a similar reasoning).

 \begin{lemma}\label{main_result2}[Convergence rates of $\{(x_{k},y_k,\lambda_{k})\}_{k\geq1}$] Let  $\{z_k:=(x_k,y_k, \lambda_k)\}_{k\geq1}$ be the sequence generated by Algorithm \ref{alg1}. If Assumption  \ref{assum:bounded_iter} holds,  Assumption \ref{assump2} is satisfied on a compact set where the iterates belong to,  $\{P_{k}\}_{k\geq1}$ is defined as in \eqref{lyapunov_function}, \redd{ the function $g$ is continuous},   $\rho$ is chosen as in Lemma \ref{decrease} and, additionally,  $ P$ defined in \eqref{P} satisfies the KL property  at $u^{*}:=(x^{*},y^*,\lambda^{*},y^{*},\gamma^{*})$, where $z^{*}:=(x^{*},y^*,\lambda^{*})$ is the limit point of $\{z_k\}_{k\geq1}$ and $\gamma^{*}$ is a limit point of  $\{\gamma_k\}_{k\geq1}$, then there exists  $k_1\geq1$ such that for all $k\geq k_1$ we have:
       \begin{equation*}\label{rate_point}
           \|z_{k}-z^{*}\|\leq C\max\{\varphi(\mathcal{E}_{k}),\sqrt{\mathcal{E}_{k-1}}\},
       \end{equation*}
       where $C>0$  and  $\varphi\in\Psi_{\tau}$, with $\tau>0$, denotes a desingularizing function. 
 \end{lemma}
 
 \begin{proof}
 From Lemma \ref{decrease}, the sequence $\{P_k\}_{k\geq1}$ is monotonically decreasing, and consequently $\{\mathcal{E}_{k}\}_{k\geq 1}$ is  monotonically decreasing. 
 Using \eqref{llyap} and the fact that  $\{\mathcal{E}_{k}\}_{k\geq 1}$ is non-negative, we have for all $k\geq 1$: 
\begin{align}\label{lmit2}
    \|\Delta x_{k+1}\|+ \|\Delta y_{k+1}\|+\|\Delta\lambda_{k+1}\|\leq 8\sqrt{\frac{2\frac{(\theta+\red{\max\{L_h,\kappa\}})^2}{\sigma^2}+1}{\ubar{\gamma}}}\sqrt{\mathcal{E}_{k}}.
\end{align}
Without loss of generality, we assume that $\gamma^{*}$ is unique. Since $ P_{k}\to P^{*}$, ${u_k}\to u^{*}$ and $P$ satisfies the KL property at $u^{*}$, then there exists   $k_1=k_1(\epsilon,\tau)\geq 1$   such that $\forall k>k_1$, we have $\|u_k-u^{*}\|\leq \epsilon$ and $P^{*}<P_{k}<P^{*}+\tau$, and  the following KL property holds:
 \begin{equation}\label{KL1}
     \varphi'(\mathcal{E}_{k})\|\partial P(x_{k},y_k,\lambda_{k},y_{k-1},\gamma_k)\|\geq1.
 \end{equation}
Hence, using the same argument as in \textbf{Case 2} in the proof of Lemma \ref{finite_length}, \eqref{used_later} follows:
\begin{align*}
    \|\Delta x_{k+1}\|+\|\Delta y_{k+1}\|+\|\Delta\lambda_{k+1}\|
    \leq&\frac{32\left(2\frac{(\theta+\red{\max\{L_h,\kappa\}})^2}{\sigma^2}+1\right)\eta}{\ubar{\gamma}}\Big(\varphi(\mathcal{E}_{k})-\varphi(\mathcal{E}_{k+1})\Big)\\
    &+\frac{2c+\theta+D_S}{\eta}\left(\|\Delta x_{k}\|+\|\Delta y_{k}\|+\|\Delta\lambda_{k}\|\right).
\end{align*}

\noindent Let us now choose $\eta>0$ such that $0<\frac{2c+\theta+D_S}{\eta}<1$ and define a parameter $\delta_0$ as  $\delta_0=1-\frac{2c+\theta+D_S}{\eta}>0$. Then, 
summing up the above inequality over $k>k_1$, we get: 
\begin{align*}
   \sum_{k\geq k_1}{\|\Delta x_{k+1}\|+\|\Delta y_{k+1}\|+\|\Delta\lambda_{k+1}\|}
  \leq&\frac{32\left(2\frac{(\theta+\red{\max\{L_h,\kappa\}})^2}{\sigma^2}+1\right)\eta}{\ubar{\gamma}\delta_0}\varphi(\mathcal{E}_{{k_1}})\nonumber\\
    &+\frac{2c+\theta+D_S}{\eta\delta_0}\Big(\|\Delta x_{k_1}\|+\|\Delta y_{k_1}\|+\|\Delta\lambda_{k_1}\|\Big).
\end{align*}
Hence, using the triangle inequality, we get for any $k \geq k_1$:
 \begin{align*}
  \|z_{k}-z^{*}\| &\leq \sum_{l\geq k}{\|z_{l}-z_{l+1}\|}\leq  \sum_{l\geq k}{\|\Delta x_{l+1}\|+\|\Delta y_{l+1}\|+\|\Delta\lambda_{l+1}\|}\\
 &\leq \! \frac{32\! \left(2\frac{(\theta+\red{\max\{L_h,\kappa\}})^2}{\sigma^2} \!+\!1 \right)\eta}{\ubar{\gamma}\delta_0}\varphi(\mathcal{E}_{{k}}) \!+\! \frac{2c \!+\theta \!+D_S}{\eta\delta_0}\Big(\|\Delta x_{k}\| \!+\|\Delta y_{k}\| \!+\|\Delta\lambda_{k}\|\Big).
 \end{align*}
 Further, using \eqref{lmit2}, it follows that:
  \begin{align*}
 &\|z_{k}-z^{*}\|\\
 \leq&\frac{32\left(2\frac{(\theta+\red{\max\{L_h,\kappa\}})^2}{\sigma^2}+1\right)\eta}{\ubar{\gamma}\delta_0}\varphi(\mathcal{E}_{{k}}) 
+\frac{8(2c+\theta+D_S)}{\eta\delta_0}\sqrt{\frac{2\frac{(\theta+\red{\max\{L_h,\kappa\}})^2}{\sigma^2}+1}{\ubar{\gamma}}}\sqrt{\mathcal{E}_{k-1}}\\
    \leq& C \max\{\varphi(\mathcal{E}_{{k}}),\sqrt{\mathcal{E}_{k-1}}\},
 \end{align*}
 where 
 \begin{align*}
    C=\max\left\{ \frac{32\left(2\frac{(\theta+\red{\max\{L_h,\kappa\}})^2}{\sigma^2}+1\right)\eta}{\ubar{\gamma}\delta_0},\right.\left. \quad \frac{8(2c+\theta+D_S)}{\eta\delta_0}\sqrt{\frac{2\frac{(\theta+\red{\max\{L_h,\kappa\}})^2}{\sigma^2}+1}{\ubar{\gamma}}}\right\}.
 \end{align*}
 This concludes our proof.
 \end{proof}

\medskip 

 \noindent The following theorem derives the convergence rate of the sequence generated by Algorithm \ref{alg1} when the Lyapunov function satisfies the KL property with the special desingularizing function  $\varphi$ (this is the case when e.g., $P$ is semi-algebraic \cite{BolDan:07}):
 \[
 \varphi:[0,\tau)\to[0,+\infty),\; \varphi(s)=s^{1-\nu}, \text{ where } \nu\in[0,1).
 \]
 \begin{theorem}\label{corollary}[Convergence rates of $\{(x_{k},y_k,\lambda_{k})\}_{k\geq1}$] Let assumptions of Lemma \ref{main_result2} hold and    $z^{*}:=(x^{*},y^*,\lambda^{*})$ be the limit point of the sequence $\{z_k:=(x_{k},y_k,\lambda_{k})\}_{k\geq1}$ generated by Algorithm \ref{alg1}. If $P$ satisfies the KL property at $u^{*}:=(x^{*},y^*,\lambda^{*},y^{*},\gamma^{*})$, where $\gamma^{*}$ is a limit point of the sequence $\{\gamma_k\}_{k\geq1}$, with the following special  desingularizing function:
       \[
       \varphi:[0,\tau)\to[0,+\infty),\; \varphi(s)=s^{1-\nu}, \text{ where } \nu\in[0,1),
       \]
       then the following rates hold:
       \begin{enumerate}
         \item If $\nu=0$, then $z_{k}$ converges to $z^{*}$ in a finite number of iterations.
         \item If $\nu\in(0,\frac{1}{2}]$, then for all $k\geq k_1 $, we have:
         \[
         \|z_{k}-z^{*}\|\leq\frac{\sqrt{\mathcal{E}_{k_1}}}{\sqrt{(1+\bar{c}\mathcal{E}_{k_1}^{2\nu-1}})^{k-k_1}}, 
         \]
  where \;\;$ \bar{c}=\frac{\ubar{\gamma}}{4(2c+\theta+D_S)^2}$.    
         \item If $\nu\in(\frac{1}{2},1)$, then for all $k> k_1 $, we have:
         \[
         \|z_{k}-z^{*}\|\leq\left(\frac{1}{\mu(k-k_1)+\mathcal{E}_{k_1}^{1-2\nu}}\right)^{\frac{1-\nu}{2\nu-1}}.
         \]
       \end{enumerate}
 \end{theorem}
 
 \begin{proof}
 Let $ \nu\in[0,1)$ and for all  $ s\in [0,\tau), \varphi(s)=s^{1-\nu}$ and  $\varphi'(s)=(1-\nu)s^{-\nu}$.  It follows that $\forall k\geq k_1$, we have:
        \begin{equation}\label{rate_point1}
           \|z_{k}-z^{*}\|\leq C\max\{\mathcal{E}_{k}^{1-\nu},\sqrt{\mathcal{E}_{k-1}}\}.
       \end{equation}
Furthermore,  \redd{\eqref{KL1}} yields:
       \[ \mathcal{E}_k^{\redd{\nu}}\leq \|\partial P(x_{k},\lambda_{k},x_{k-1},\gamma_k)\| \hspace{0.5cm} \forall k\geq k_1.\]
Moreover, from \eqref{rate} and Lemma \ref{bounded_grad}, there exists $p_k\in\partial P(x_{k},\lambda_{k},x_{k-1},\gamma_k)$ such that for any $k\geq 1$, we have:
 \[
        \|p_k\|^2\leq\frac{4(2c+\theta+D_S)^2}{\ubar{\gamma}}(\mathcal{E}_{k-1}-\mathcal{E}_{k}).
 \]
 Hence, 
    \[ \mathcal{E}_k^{2\nu}\leq \frac{4(2c+\theta+D_S)^2}{\ubar{\gamma}}(\mathcal{E}_{k-1}-\mathcal{E}_{k}) \hspace{0.5cm} \forall k> k_1.\]
Setting $\bar{c}=\frac{\ubar{\gamma}}{4(2c+\theta+D_S)^2}>0,$
 we get the recurrence  \[ \bar{c}\mathcal{E}_k^{2\nu}\leq\mathcal{E}_{k-1}-\mathcal{E}_{k} \hspace{0.5cm} \forall k> k_1. \]
 \begin{enumerate}
         \item Let $\nu=0$. If $\mathcal{E}_k>0$ for any $k> k_1$, we have $\bar{c}\leq \mathcal{E}_{k-1}-\mathcal{E}_{k}$. As $k$ goes to infinity, the right hand side approaches zero. Then, $0<\bar{c}\leq0$ which is a contradiction. Hence,  there exists $ k> k_1 $ such that $ \mathcal{E}_k=0.$ Then, $\mathcal{E}_k\to 0$ in a finite number of steps and  from \eqref{rate_point1}, $z_k\to z^{*}$ in a finite number of steps.
         \item Let $\nu\in(0,\frac{1}{2}]$. Then, $2\nu-1\leq0$.
         Let $k > k_1$. Since $\{\mathcal{E}_i\}_{i\geq k_1}$ is monotonically decreasing, then $\mathcal{E}_i\leq\mathcal{E}_{k_1}$ for any $i\in\{k_1+1, k_1+2,..., k\}$ and 
         \[\bar{c}\mathcal{E}_{k_1}^{2\nu-1}\mathcal{E}_k\leq\mathcal{E}_{k-1}-\mathcal{E}_{k} \hspace{0.5cm} \forall k> k_1.\]
         Rearranging this, we get  for all $k> k_1$:
       \begin{align*}
       \mathcal{E}_k\leq \frac{\mathcal{E}_{k-1}}{1+\bar{c}\mathcal{E}_{k_1}^{2\nu-1}}\leq\frac{\mathcal{E}_{k-2}}{(1+\bar{c}\mathcal{E}_{k_1}^{2\nu-1})^2}\leq \cdots  \leq\frac{\mathcal{E}_{k_1}}{(1+\bar{c}\mathcal{E}_{k_1}^{2\nu-1})^{k-k_1}}.
       \end{align*}
       Then, we have $\max\{\mathcal{E}_k^{1-\nu},\sqrt{\mathcal{E}_{k-1}}\}=\sqrt{\mathcal{E}_{k-1}}.$
        It then follows that:
        \[
         \|z_{k}-z^{*}\|\leq\frac{\sqrt{\mathcal{E}_{k_1}}}{\sqrt{(1+\bar{c}\mathcal{E}_{k_1}^{2\nu-1}})^{k-k_1}},
         \]
       \item Let $\nu\in(1/2,1)$, we have: 
       \begin{equation}\label{eqqq}
           \bar{c}\leq(\mathcal{E}_{k-1}-\mathcal{E}_k)\mathcal{E}_k^{-2\nu} \hspace{0.5cm} \forall k> k_1.
       \end{equation}
    Let $h:\mathbb{R}_{+}\to\mathbb{R}$ be defined as $h(s)=s^{-2\nu}$ for any $s\in\mathbb{R}_{+}$. It is clear that $h$ is monotonically decreasing and $\forall s\in\mathbb{R}_+, h'(s)=-2\nu s^{-(1+2\nu)}<0$. Since $\mathcal{E}_k\leq\mathcal{E}_{k-1}$ for all $k> k_1$, then $h(\mathcal{E}_{k-1})\leq h(\mathcal{E}_k)$ for all $k> k_1$. We consider two cases:
    \newline
\textbf{{Case 1}}: Let $r_0\in(1,+\infty)$ such that:
$ h(\mathcal{E}_k)\leq r_0h(\mathcal{E}_{k-1}), \; \forall k> k_1.$
Then, from  \eqref{eqqq} we get:
\begin{align*}
 &   \bar{c}\leq r_0(\mathcal{E}_{k-1}-\mathcal{E}_k)h(\mathcal{E}_{k-1})\leq r_0h(\mathcal{E}_{k-1})\int_{\mathcal{E}_k}^{\mathcal{E}_{k-1}}{1\,ds}\\
    &\leq r_0\int_{\mathcal{E}_k}^{\mathcal{E}_{k-1}}{h(s)\,ds}= r_0\int_{\mathcal{E}_k}^{\mathcal{E}_{k-1}}{s^{-2\nu}\,ds}=\frac{r_0}{1-2\nu}(\mathcal{E}_{k-1}^{1-2\nu}-\mathcal{E}_{k}^{1-2\nu}).
\end{align*}
Since $\nu>\frac{1}{2}$, it follows that:
\[
0<\frac{\bar{c}(2\nu-1)}{r_0}\leq \mathcal{E}_{k}^{1-2\nu}-\mathcal{E}_{k-1}^{1-2\nu}.
\]

Let us define $\hat{c}=\frac{\bar{c}(2\nu-1)}{r_0}$ and $\hat{\nu}=1-2\nu<0$. We get:
\begin{equation}\label{need1}
    0<\hat{c}\leq \mathcal{E}_{k}^{\hat{\nu}}-\mathcal{E}_{k-1}^{\hat{\nu}} \hspace{0.5cm} \forall k> k_1.
\end{equation}
    \newline
\textbf{{Case 2}}: Let $r_0\in(1,+\infty)$ such that:
$h(\mathcal{E}_k)> r_0h(\mathcal{E}_{k-1}), \;  k> k_1$. We then have $\mathcal{E}_k^{-2\nu}\geq r_0\mathcal{E}_{k-1}^{-2\nu}$. This leads to 
\[
q\mathcal{E}_{k-1}\geq\mathcal{E}_k,
\]
where $q={r_0}^{-\frac{1}{2\nu}}\in(0,1)$. Since $\hat{\nu}=1-2\nu<0$ we have $ q^{\hat{\nu}}\mathcal{E}_{k-1}^{\hat{\nu}}\leq\mathcal{E}_k^{\hat{\nu}}$ and then, it follows that: 
\[(q^{\hat{\nu}}-1)\mathcal{E}_{k-1}^{\hat{\nu}}\leq\mathcal{E}_{k-1}^{\hat{\nu}}-\mathcal{E}_{k}^{\hat{\nu}}.
\] 
Since $q^{\hat{\nu}}-1>0$ and $\mathcal{E}_k\to0^+$ as $k\to\infty$, there exists $\Tilde{c}$ such that $(q^{\hat{\nu}}-1)\mathcal{E}_{k-1}^{\hat{\nu}}\geq\tilde{c}$ for all $k> k_1$. Therefore, we obtain:
\begin{equation}\label{need2}
        0<\tilde{c}\leq \mathcal{E}_{k}^{\hat{\nu}}-\mathcal{E}_{k-1}^{\hat{\nu}} \hspace{0.5cm} \forall k> k_1.
\end{equation}
By choosing ${\mu}=\min\{\hat{c},\tilde{c}\}>0$, one can combine \eqref{need1} and \eqref{need2} to obtain
\[
        0<{\mu}\leq \mathcal{E}_{k}^{\hat{\nu}}-\mathcal{E}_{k-1}^{\hat{\nu}} \hspace{0.5cm} \forall k> k_1.
\]
Summing the above inequality from $k_1+1$ to some $k> k_1$ gives
\[
\mu(k-k_1) + \mathcal{E}_{k_1}^{\hat{\nu}}\leq  \mathcal{E}_{k}^{\hat{\nu}}.
\]
Hence, \[
\mathcal{E}_k\leq (\mu(k-k_1) + \mathcal{E}_{k_1}^{\hat{\nu}})^{\frac{1}{\hat{\nu}}}=(\mu(k-k_1) + \mathcal{E}_{k_1}^{1-2{\nu}})^{\frac{1}{1-2{\nu}}}.
\]
Since $\nu\in(\frac{1}{2},1)$, then $\max\{\mathcal{E}_{k-1}^{1-\nu},\sqrt{\mathcal{E}_{k-1}}\}=\mathcal{E}_{k-1}^{1-\nu}.$ Then, \eqref{rate_point1} becomes: 
   \[
         \|z_{k}-z^{*}\|\leq\left(\frac{1}{\mu(k-k_1)+\mathcal{E}_{k_1}^{1-2\nu}}\right)^{\frac{1-\nu}{2\nu-1}}, \hspace{0.2cm} \forall k>k_1.
         \]
   \end{enumerate} 
This concludes our proof.
\end{proof}

\medskip

\noindent  Note that our convergence analysis under KL is similar to that found in the literature, see  e.g., \cite{ BolSab:18, Yas:22, CohHal:21} (although the convergence analysis under the KL property was also addressed in \cite{CohHal:21}, explicit rates associated with this property were not provided there). In conclusion, in addition to its straightforward implementation and  simplicity of iteration steps, our algorithm iL-ADMM enjoys mathematical guarantees of  convergence, ensuring that it can reliably find optimal solutions to a wide range of nonconvex problems.


 \subsection{Selection of  penalty parameter $\rho$}

\noindent \redd{The previous convergence results  rely on the assumption that the penalty parameter $\rho$ exceeds a certain threshold, see \eqref{gamma_rho}. However, in practice, determining this threshold beforehand poses challenges as it depends on unknown parameters of the  problem's data and the algorithm's parameters. 
\begin{algorithm} 
\caption{iL-ADMM method with trial values of $\rho$}
\label{backtracking_alg}
\begin{algorithmic}[1]
\State \textbf{Initialization:} Choose $(x_{-1}^{*}, y_{-1}^{*}, \lambda_{-1}^{*}) \in \mathbb{R}^n \times \mathbb{R}^p \times \mathbb{R}^m$, $\zeta_1, \zeta_2 > 1$, $\epsilon > 0$, $\rho_0 > 0$, and $K_0 > 0$.
\State $t \gets 0$
\While{$\epsilon$-KKT conditions are not satisfied}
    \State Call  \textbf{Algorithm \ref{alg1}} with $\rho = \rho_t$ 
 and warm start $(x_0, y_0, \lambda_0) \gets (x_{t-1}^{*}, y_{t-1}^{*}, \lambda_{t-1}^{*})$ 
 \Statex{~~~~}  for $K_t$  iterations, yielding $(x_{K_t}, y_{K_t}, \lambda_{K_t})$.
 \State Update $(x_{t}^{*}, y_{t}^{*}, \lambda_{t}^{*}) \gets (x_{K_t}, y_{K_t}, \lambda_{K_t})$.
    \State Update  $K_{t+1} \gets \zeta_1 K_t$ and $\rho_{t+1} \gets \zeta_2 \rho_t$.
    \State $t \gets t+1$
\EndWhile
\end{algorithmic}
\end{algorithm}}
\redd{To overcome this challenge, we propose in this section an outer algorithm that repeatedly calls Algorithm \ref{alg1} for a fixed number of iterations, denoted as $K_t$, using a penalty parameter $\rho_t$. If Algorithm 1 does not yield an $\epsilon$-KKT point for the problem \eqref{eq1} within $K_t$ iterations, then both $K_t$ and $\rho_t$ are increased geometrically. Specifically, we set $K_{t+1} = \zeta_1 K_t$ and $\rho_{t+1} = \zeta_2 \rho_t$, where $\zeta_1, \zeta_2 > 1$. The resulting procedure can be summarized in Algorithm \ref{backtracking_alg}.}

\medskip 

\noindent \redd{This approach has been also used, e.g., in \cite{XieWri:21}. Following similar arguments as in \cite{XieWri:21}, we can prove that the above algorithm is well-defined and  it yields an $\epsilon$-KKT point of problem \eqref{eq1} in a finite number of calls of Algorithm \ref{alg1}.
}


\section{Numerical results}\label{sec5}
\noindent In this section, we compare  iL-ADMM algorithm  with the dynamic linearized alternating direction method of multipliers (DAM) from \cite{CohHal:21} and the solver IPOPT \cite{WacBie:06} for solving nonlinear model predictive control  and matrix factorization problems using real dynamical systems and  datasets, respectively. The implementation details are conducted using MATLAB  on a laptop equipped with an i7 CPU operating at 2.9 GHz and 16 GB of RAM.

\subsection{Nonlinear model predictive control}
In this section, we consider nonlinear model predictive control (NMPC) problems  for several nonlinear systems:  inverted pendulum on a cart (IPOC) system \cite{BreLua:19},  single machine infinite bus  (SMIB) system  \cite{ThaAld:14},  lane tracking  (LT)  system from MathWorks' MPC toolbox, four tanks (4T) system  \cite{RafHub:06}, and free-flying robot (FFR) system \cite{Sak:99}. For a continuous  nonlinear system  we employ Euler discretization with a sampling time $T$ to obtain a discrete-time model of the form:
\[
{z}(t+1):=\psi({z}(t),{u}(t)),
\]
where  ${u}\in\mathbb{R}^{i_d}$  denotes the inputs  and ${z}\in\mathbb{R}^{s_d}$ the states. For all systems, we consider input  constraints of the form:
\[
{u}_\text{min}\leq {u}(t)\leq {u}_\text{max}. 
\]

\noindent Our goal is to drive the system to a desired state  ${z}_{\text{e}}$  and input $u_{\text{e}}$. To achieve this, we apply a NMPC scheme. To formulate the NMPC problem  as a  nonconvex  optimization problem, we   adopt a single shooting approach, where the state variables are eliminated under the assumption of a piecewise constant input trajectory. Then, we use auxiliary variables to equate the states of the system at hand. The decision variables for NMPC are given by $x = ({u}(0),\cdots,{u}(N-1)) \in \mathbb{R}^{N i_d}$, where $N$ is the prediction horizon. If we introduce a sequence of functions $F_j: \mathbb{R}^{Ni_d} \rightarrow \mathbb{R}^{s_d}$ defined~as
\[
F_0(x) = {z}(0), \;\;
F_{j+1}(x) = \psi(F_j(x),{u}(j)), \quad j=0:N-1,
\]
then the resulting NMPC problem that needs to be solved at each sampling time is given by:
\begin{align}
\label{MPC_pb}
& \min_{\left(x,\{y^{j+1}\}_{j=0}^{N-1}\right)} \varphi(x,y):= \frac{1}{2}\sum_{j=0}^{N-1}\| y^{j+1} - {z}_{\text{e}} \|_Q^2 + \|{u}(j)-u_{\text{e}}\|_R^2  \nonumber \\
& \hspace{1cm}\textrm{s.t.:} \hspace{0.5cm}  F_{j+1}(x)-y^{j+1}=0,\;\;  {z}({0}) \;  \text{given}\\ 
&  \hspace{2.2cm}{u}_\text{min}\leq x \leq{u}_\text{max}, \;  j=0:N-1,  \nonumber
\end{align}
where the matrices $Q, R \succeq 0$ and we used the notation $\|z\|_Q^2 = z^T Q z$.  The nonconvex problem described in  \eqref{MPC_pb} can be reformulated as problem \eqref{eq1}, where $-G$ is the identity matrix of dimension $N s_d$. The smooth functions $f$ and $h$ are convex quadratic,  the nonsmooth function $g$ is the indicator function of the set describing the  input box constraints and $\mathcal{Y} = \mathbb{R}^{N s_d}$. \redd{At this point, it is worth mentioning that since $g$ is the indicator function of a box set, then Step 4 of iL-ADMM reduces to finding a solution of a strongly convex QP  with box constraints, which is solved with \texttt{quadprog} from Matlab, and its counterpart in DAM \cite{CohHal:21} basically reduce to computing a projection onto some box constraints. On the other hand, Step 5 in iL-ADMM and its counterpart in DAM reduces to a gradient step. } 

\medskip 

\noindent  For simulations, the parameters of the systems, of the constraints and of the stage costs (matrices $Q,R$) are taken as in the cited references for each system. For NMPC we used the setup from Table 1 (here $N_{\text{sim}}$  denotes the simulation horizon, i.e., the number of times we solve the NMPC optimization problems at different $z(0)$'s associated with each system). \redd{We initialize all the algorithms, in the first step of NMPC, randomly, while in the subsequent NMPC steps we use a warm start strategy, i.e., we use the solution of the current NMPC step as the initialization for the algorithms in the next NMPC step. In our simulations, we stop iL-ADMM and DAM  when  $\|\mathcal{F}_{k}\| \leq 10^{-6}$ and $|\varphi_{k} - \varphi_{k-1}| \leq 10^{-5}$, where $\mathcal{F}_k$ and $\varphi_k$ denote the functional constraints and the objective function evaluated at the current iterate $(x_k,y_k)$, respectively. } 
\begin{table}
  \renewcommand{\arraystretch}{1}
\begin{adjustbox}{width=0.7\columnwidth,center}
  \begin{tabular}{|c|c|c|c|c|}
    \hline
    \backslashbox{System}{Parameters} &   $N_{\text{sim}}$ &  $N$ & $T$ &$z(0)$ \\
    \hline
    SMIB  & $50$ & $10$ & $0.01$&$[0.05,0.1,0.1,0.2]^T $\\
    \hline
   IPOC & $40$& $10$ & $0.1$&$ [0,0,0.5,0]^T$\\
    \hline
    4T & $100$& $20$ & $3$&$[20,20,20,20]^T $\\
    \hline
    LT & $50$& $10$ & $0.1$&$ [0.1,0.5,25,0.1,0.1,0.001,0]^T$\\
    \hline
    FFR & $50$& $30$ & $0.4$&$[-10,-10,\frac{\pi}{2},0,0,0]^T $\\
    \hline
    \end{tabular}
       \end{adjustbox}
      \label{tab20}
      \caption{Systems and nonlinear MPC setup.}
\end{table}

\begin{table}  
\small
  \centering
    \renewcommand{\arraystretch}{1}
    \begin{adjustbox}{width=0.35\columnwidth,center}
  \begin{tabular}{|c|c|c|c|}
    \hline
    \backslashbox{Method}{Parameters} &   $\rho$ &  $\beta_k$ & $\theta_k$  \\
    \hline
      DAM from \cite{CohHal:21}  & $3$ & $10$ & $1$\\
    \hline
    iL-ADMM & $5$& $1$ & $1$\\
    \hline
    \end{tabular}
    \end{adjustbox}
      \caption{ Parameters for iL-ADMM and DAM used in NMPC.}
\end{table}

\medskip 

\noindent Moreover, after some search we found that  iL-ADMM and DAM  algorithms are performing well on all test systems with the parameter choices  from Table 2.   Note that the same parameters are used to solve all the  NMPC problems. For DAM, a larger value of $\beta_k$ is required to cover the big approximation error generated by the linearization of the full smooth part of the augmented Lagrangian function. On the other hand, we chose $\theta_k=1$ for both  methods. This table shows the robustness of our method iL-ADMM w.r.t. its parameters, as it requires minimum tuning. 

\medskip 

\noindent In Table 3, we report for each system the average number of iterations, $E(\# \text{iter})$, required for each algorithm to solve the nonlinear MPC problems over the simulation horizon $N_{\text{sim}}$ and the standard deviation, $\sigma(\# \text{iter})$\footnote{Standard deviation is computed as: $\sigma(\#\text{iter}) = \left(1/N_\text{sim} \sum_{i=1}^{N_\text{sim}}  (\#\text{iter}(i) - E(\# \text{iter}))^2 \right)^{0.5}$.}; similarly, the average CPU time (in seconds), $E(\text{cpu})$, and the corresponding standard deviation, $\sigma(\text{cpu})$;  the optimal value, $\varphi^*$, found by each algorithm for the first NMPC problem for each system, and the corresponding  infeasibility  $\|\mathcal{F}\|$. As can be seen from Table 3, in comparison to DAM, our algorithm iL-ADMM requires fewer iterations to solve the problem since the model considered in Step 4 approximates the original augmented Lagrangian better than the one considered in DAM, resulting in our algorithm being  faster than DAM in terms of cpu time. When compared with IPOPT, our method iL-ADMM appears superior in terms of CPU time for all systems except for one, the four-tank system, where IPOPT finds a solution in less CPU time, but still comparable to our method. However, for the number of iterations, IPOPT consistently requires fewer iterations than our method. We attribute this to the fact that IPOPT uses second derivative information, but, on the other hand, this necessitates more time to evaluate them. For all the systems, Table 3 clearly indicates that both, the proposed method iL-ADMM and  DAM algorithm achieve an optimal value for the first NMPC problem that is very close (generally coinciding) to that of the IPOPT solver. 

\begin{table}
\small
{ 
\centering 
  \renewcommand{\arraystretch}{1.7}
\begin{adjustbox}{width=\columnwidth,center}
    \begin{tabular}{|c|ccc|ccc|ccc|}
    \hline
   \multirow{3}{*}{\backslashbox{\!\!System $(i_d,s_d)$}{Algorithm}} 
     & \multicolumn{3}{c|}{iL-ADMM} &
     \multicolumn{3}{c|}{DAM \cite{CohHal:21}} &
      \multicolumn{3}{c|}{IPOPT} \\ 
              & E(\# iter)     & E(cpu) & $\varphi^*$
          &   E(\# iter)    & E(cpu) & $\varphi^*$ &
           E(\# iter)     &  E(cpu) & $\varphi^*$ \\ 
         &  $\sigma$(\# iter)   &$\sigma$(cpu)   & $\|\mathcal{F}\|$  
          &   $\sigma$(\# iter)   & $\sigma$(cpu)&   $\|\mathcal{F}\|$
           &  $\sigma$(\# iter)   &$\sigma$(cpu) & $\|\mathcal{F}\|$ \\
    \hline
    
    SMIB & 548.24 &  \textbf{0.98} & 0.3087
    & 3877.83 & 9.95  & 0.3089&
      87.28& 1.12&0.3087 \\
    (2,4) & 0.09&  1.63e-4&9.91e-7
      &0.22 & 0.01&9.99e-7 & 
      6e-2& 1e-3&2.87e-8\\\hline

          IPOC & 153.74 &  \textbf{0.36}& 166.60
    &576.36 & 17.41  & 166.60&
      28.7 &0.57 & 166.60\\
   (1,4) & 0.03&  1.23e-4&7.56e-7
      &0.08 & 5.57e-3&8.01e-7 & 
      2.3e-2&1.04e-3 &3.69e-8\\\hline

    4T & 1643.72 & 4.37 & 87.94
    &6464.35 &  15.03 & 87.94&
      187.64& \textbf{2.79}&87.94 \\
   (2,6) & 0.02&  2.1e-3& 2.45e-7
      &0.38 & 0.07&6.34e-8 & 
      2.66e-3& 4.36e-2&3.68e-8\\\hline

   LT & 169.46 &  \textbf{0.82}& 6.98
    &732.06 &  7.74 & 6.98&
      32.94& 1.03&6.98 \\
   (2,7) & 0.007&  1.75e-4&1.42e-7
      &0.14 & 0.008&9.36e-7 & 
      0.005& 1.24e-4&4.93e-8\\\hline

     FFR & 102.01 &  \textbf{1.33}& 1066.24
    & 1877.13 &  12.76 & 1066.42&
      23.94& 2.88&1065.87 \\
for each system  (2,6)  & 0.01&  2.84e-5& 9.74e-7
      &0.17 & 0.01&9.94e-7 & 
      0.02& 1.69e-4&1.87e-7\\
      \hline
    \end{tabular}%
    \end{adjustbox}
}
 \caption{ Numerical results comparing iL-ADMM, DAM and IPOPT on solving $N_{\text{sim}}$ nonlinear MPC problems for $5$  dynamical systems of different dimensions.}
\label{tab1}
\end{table}

\medskip 

\begin{figure}[htp]
   \begin{center}
       \includegraphics[width=\textwidth,height=8cm]{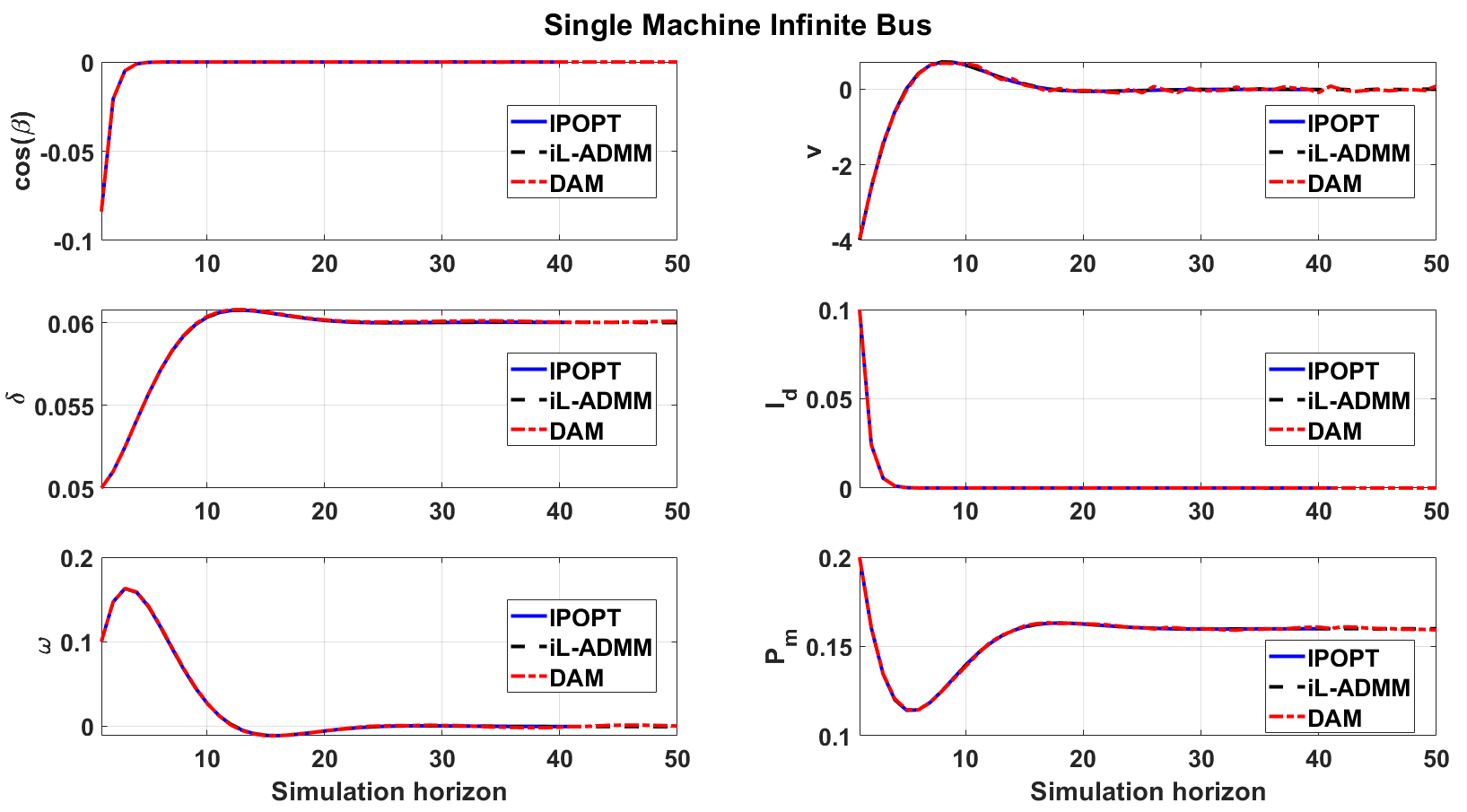}
    \caption{Closed-loop nonlinear MPC trajectories over a simulation horizon $N_\text{sim}=50$ for the single machine infinite bus system    computed using iL-ADMM, DAM and IPOPT (2 inputs, 4 states).}
    \label{fig2}   
   \end{center}
\end{figure}

\noindent Figure \ref{fig2} illustrates the closed-loop NMPC trajectories of inputs and states over a simulation horizon of lenght  $N_\text{sim}=50$ for the single machine infinite bus system obtained using iL-ADMM (Algorithm \ref{alg1}), DAM, and IPOPT. The figure distinctly demonstrates the success of all methods in stabilizing the system. Moreover, it reveals that these methods yield (nearly) identical solutions to the optimization problem \eqref{MPC_pb} for different $z(0)$'s over $N_\text{sim}$, as indicated by the similarities in the  closed-loop  NMPC trajectories among all three methods.


\subsection{Nonnegative orthogonal matrix factorization}
\redd{In this section, we consider factorizing a nonnegative matrix as a product of two nonnegative matrices and, additionally, one is required to be  orthogonal (note that orthogonal and nonnegative constraints lead to sparsity in the corresponding matrix).  This problem can be formulated as follows \cite{Gil:20}:  
\begin{equation} 
\label{matrix_fact}
\min_{U, V \geq 0} \quad \frac{1}{2} \| A - U V^\top \|_F^2 + \frac{\gamma}{2} \| V^\top V - I_r \|_F^2,
\end{equation}
where $A \in \mathbb{R}^{N \times d}_+$ is a given  hyperspectral image with $N$ denoting the number of pixels and $d$ the number of spectral bands,  \( U \in \mathbb{R}^{N \times r}_+ \) and \( V \in \mathbb{R}^{d \times r}_+ \) are  nonnegative matrices  and \( r \) is the latent dimension (rank of the factorization). To smoothly enforce orthogonality on \( V \), we introduce a regularization term controlled by a parameter \( \gamma >0 \).  Adding a slack variable $y$, problem \eqref{matrix_fact}  can be formulated as a particular case of  \eqref{eq1}, where $x= (U, V)$, $f(x) = \frac{\gamma}{2} \| V^\top V - I_r \|_F^2$,  $g$ is the indicator function of the positive orthant, $F(x) = UV^\top$,  $h(y) = \frac{1}{2} \| A - y \|_F^2$, $G = -I$ (hence, $m=p$) and $\mathcal{Y}$ is the full space. This reformulation allows us to leverage our proposed algorithm and DAM to solve the nonnegative orthogonal matrix factorization problem.  On the other hand, since  this reformulation yields many constraints and since IPOPT has difficulties when solving problems with many constraints, we use IPOPT to solve the original problem \eqref{matrix_fact}. 
}
\medskip 

\noindent \redd{In the experiments, we consider two values for the factorization rank, \( r=3 \) and \( r=10 \), and two values of the orthogonality parameter, \( \gamma = 10^{-2} \) and \( \gamma = 10^2 \).  We initialize all algorithms at the same randomly chosen starting point. The stopping criteria for iL-ADMM and DAM are  the following $\epsilon$-KKT conditions:  
\begin{align*}
    &\|F(x_k) - y_k\| \leq 10^{-3}\\
&\dist\left(-\nabla f(x_{k}) - {\nabla F(x_{k})}^T\lambda_{k}, \partial g(x_{k})\right) + \left\|\nabla h(y_{k}) - \lambda_{k}\right\|  \leq 10^{-2}.
\end{align*}  
The subproblem in Step 4 of iL-ADMM has strongly convex quadratic objective and nonnegative constraints and it is solved using the  accelerated gradient method \cite{Nes:18}, which is stopped when condition \eqref{inexactness} is satisfied with $\alpha = 10$, while Step 5 reduces to a gradient update. The parameters \( \beta_k \) and \( \theta_k \) are chosen to satisfy the conditions in \eqref{eq_assu} and \eqref{eq_assu1}, respectively.  
For the penalty parameter \( \rho \), we experiment with different values:  $\rho_{\text{iL-ADMM}}$  equal to the lower bound  in \eqref{gamma_rho}; $\rho_{\text{iL-ADMM}}/{2}$, which does not satisfy the condition  \eqref{gamma_rho}; and $2\rho_{\text{iL-ADMM}}$,  to assess the robustness of our algorithm with respect to \( \rho \). Similarly, for DAM, we generate \( \beta_k \) dynamically and set \( \theta_k = 2L_h = 2 \), see (4.5) and (4.27)  in \cite{CohHal:21}, respectively. For \( \rho \) we select the lower bound from condition  (4.27) in \cite{CohHal:21}, along with $\rho_{\text{DAM}}/{2}$, which does not satisfy the condition (4.27), and $2\rho_{\text{DAM}}$.}

\medskip 

\noindent \redd{For numerical simulations we consider a hyperspectral remote sensing scene, \texttt{Salinas} dataset, taken from   \cite{Uni:00}. We pre-process  \texttt{Salinas} dataset, reducing the spatial dimension using filtering, while preserving the number of spectral bands.  The results of the numerical  experiments are presented in Table 4, which shows  the number of iterations (\# iter), CPU time in sec. (cpu), the optimal objective value (\(\varphi^*\)), and the norm of the functional constraints (\(\|\mathcal{F}\|\)) for iL-ADMM and DAM (recall that IPOPT solves the problem without functional constraints). Additionally, we include the factorization error (\(\|UV^\top - A\|_F\)) and the orthogonality error (\(\|VV^\top - I_r\|_F\)). Note that \( (n, p) \) in Table 4 refers to the dimensions of problem \eqref{eq1}, where \( n \) is the dimension of \( x \) and \( p \) is the dimension of the slack variables \( y \). In this case, we have $m=p$, where \( m \) is the number of functional constraints.  From Table 4, we observe that our algorithm outperforms DAM and IPOPT in terms of computational time. In particular, IPOPT is consistently much slower than  iL-ADMM and DAM, which we attribute to large dimension of the problem and its reliance on second derivatives. 
Furthermore, the objective values obtained by iL-ADMM and DAM are consistently better than those produced by IPOPT. We attribute this to the fact that IPOPT solves a formulation with fewer decision variables compared to iL-ADMM and DAM. We also note that iL-ADMM and DAM exhibit robustness with respect to the penalty parameter \( \rho \). As \( \rho \) increases, feasibility improves, while CPU time worsens, likely due to the fact that large values of \( \rho \) can introduce ill-conditioning in the subproblems. Clearly,  higher values of \( \gamma \) lead to a smaller orthogonality error relative to the factorization error. Finally, regarding the impact of factorization rank \( r \), we find that increasing \( r \) improves factorization performance, as indicated by a lower factorization error, but worsens the computational time.  
}

\begin{table}
\small{
\centering 
  \renewcommand{\arraystretch}{1.7}
\begin{adjustbox}{width=\columnwidth,center}
    \begin{tabular}{|c|c|c|c|ccc|ccc|ccc|}
    \hline
  \multirow{3}{*}{
   $r$} & \multirow{3}{*}{$ (n,p)$} 
 & \multirow{3}{*}{
   $\gamma$} &  \multirow{2}{*}{
   $\rho_{\text{iL-ADMM}}$}
     & \multicolumn{3}{c|}{iL-ADMM} &
     \multicolumn{3}{c|}{DAM \cite{CohHal:21}} &
      \multicolumn{3}{c|}{IPOPT} \\ \cmidrule(lr){5-13} 
           & & &    & \# iter     & $\varphi^*$ & $\|UV^T - A\|_F$
          &   \# iter    & $\varphi^*$ & $\|UV^T - A\|_F$ &
           \# iter     &  $\varphi^*$ & $\|UV^T - A\|_F$ \\ 
     & 
     & &  \multirow{1}{*}{
   $\rho_{\text{DAM}}$} &  cpu   & $\|\mathcal{F}\|$    &  $\|V^TV - I_r\|_F$
          &   cpu  & $\|\mathcal{F}\|$&  $\|V^TV - I_r\|_F$ 
          &  cpu   & - & $\|V^TV - I_r\|_F$ \\
    \hline
    
   \multirow{12}{*}{$3$}
   & \multirow{12}{*}{$(699, 2016)$}
   & \multirow{6}{*}{$10^{-2}$}   
   &\multirow{1}{*}{$144$} & 571 & 1.87  & 1.84
         & 2626  & 1.88 & 1.77 
         &  &  &  \\
     &  & &\multirow{1}{*}{$20$}
     & \textbf{1.08} & 2e-5 & 6.00
      & 1.17 & 1.7e-4& 7.78 
      &  &  &  \\ \cmidrule(lr){4-10}

       &  
   & 
   &\multirow{1}{*}{$72$} & 478 & 1.88  & 1.82
         & 2112  & 1.88 & 1.77 
         & 1567 & 48.42 & 9.84 \\
    &  & &\multirow{1}{*}{$10$}
     & \textbf{0.83} & 6.7e-5 & 6.58
      & 1.15 & 4.2e-4& 7.8 
      & 30.74 & - & 1.73 \\ \cmidrule(lr){4-10}

             & 
   & 
   &\multirow{1}{*}{$288$} & 1032 & 1.87  & 1.84
         & 2954  & 1.88 & 1.77 
         &  &  &  \\
      &  & &\multirow{1}{*}{$40$}
     & 3.21 & 5.1e-6 & 6.02
      & \textbf{1.31} & 5.7e-5& 7.82 
      &  &  &  \\ \cmidrule(lr){3-13}

         &  
   & \multirow{6}{*}{$10^2$} 
   &\multirow{1}{*}{$144$} & 910 &  50.18 & 0.32
         & 16865  & 50.18 & 0.32 
         &  &  &   \\
    &  & &\multirow{1}{*}{$20$}
     & \textbf{2.61} & 2.5e-5 & 1.00
      & 7.27 & 8.2e-5 &  1.00 
      &  &  &  \\  \cmidrule(lr){4-10}

       & 
   &  
   &\multirow{1}{*}{$72$} & 660 & 50.18  & 0.38
         & 13420  & 50.18 &  0.40
         & 445 & 539.9 & 8.93 \\
      &  & &\multirow{1}{*}{$10$}
     & \textbf{1.94} & 8.7e-5 & 1.00
      & 5.90 & 2.4e-4 &  1.00
      & 46.02 & - & 3.16 \\ \cmidrule(lr){4-10}

             &  
   & 
   &\multirow{1}{*}{$288$} & 1813 & 50.18  & 0.32
         & 17247  & 50.19 & 0.43 
         &  &  &  \\
   &  & &\multirow{1}{*}{$40$}
     & \textbf{5.16} & 6.2e-6 & 1.00
      & 7.62 & 4e-5 & 1.00  
      &  &  &  \\ \cmidrule(lr){1-13}

                \multirow{12}{*}{$10$}
         & \multirow{12}{*}{$(2330, 2016)$}
   & \multirow{6}{*}{$10^{-2}$} 
   &\multirow{1}{*}{$144$} & 1725 & 0.89  & 1.29
         & 18125  & 0.89 & 1.27 
         &  &  &  \\
      &  & &\multirow{1}{*}{$20$}
     & \textbf{4.83} & 4.1e-5 & 3.28 
     & 9.19 & 4.4e-5  & 3.98
      &  &  &  \\ \cmidrule(lr){4-10}

       & 
   & 
   &\multirow{1}{*}{$72$} & 1370 & 0.88  & 1.28
         & 17733  & 0.88 & 1.27 
         & 2409 & 15.98 & 5.64 \\
      &  & &\multirow{1}{*}{$10$}
     & \textbf{3.94} & 1e-5 & 3.36
      & 8.80 & 2.8e-4& 3.98 
      & 251.64 & - & 3.16 \\ \cmidrule(lr){4-10}

             &  
   & 
   &\multirow{1}{*}{$288$} & 3469 & 0.88  & 1.28
         & 23003  & 0.89 & 1.27 
         &  &  &  \\
   &  & &\multirow{1}{*}{$40$}
     & \textbf{9.79} & 1e-5 & 3.29
      & 10.77 & 4.2e-5& 3.98 
      &  &  &  \\ \cmidrule(lr){3-13}

         &   
   & \multirow{6}{*}{$10^2$} 
   &\multirow{1}{*}{$144$} & 3354 &  11.83 & 4.75
         & 14131  & 11.84 & 4.64
         &  &  &  \\
    &  & &\multirow{1}{*}{$20$}
     & \textbf{5.44} & 9.2e-7 & 0.1 
     & 6.47 & 3.5e-4  & 0.14
      &  &  &  \\ \cmidrule(lr){4-10}

       & 
   & 
   &\multirow{1}{*}{$72$} & 2396 & 11.83  & 4.75
         & 11970  & 11.82 & 4.63 
         & 632 & 225.88 & 12.31 \\
     &  & &\multirow{1}{*}{$10$}
     & \textbf{4.61} & 2.2e-6 & 0.1
      & 4.78 & 7.6e-4 & 0.13  
      & 12.71 & - & 1.73 \\ \cmidrule(lr){4-10}

             & 
   & 
   &\multirow{1}{*}{$288$} & 6838 & 11.83  & 4.75
         & 14811  & 11.83 & 4.64 
         &  &  &  \\
    &  & &\multirow{1}{*}{$40$}
     & 13.28 & 2.2e-7 & 0.1
      & \textbf{5.92} & 1.1e-4 & 0.14  
      &  &  &  \\ \hline

    \end{tabular}%
    \end{adjustbox}
    }
 \caption{ Numerical results comparing iL-ADMM, DAM and IPOPT on solving nonnegative orthogonal matrix factorization problems for Salinas  dataset.}
\label{tab2}
\end{table}


\section{Conclusions} 
\label{sec6}
\noindent In this paper, we introduced an inexact linearized ADMM method for solving structured nonsmooth nonconvex optimization problems. By linearizing the smooth term of the objective function and functional constraints within the augmented Lagrangian, we derived simple updates. Moreover, we solved the subproblem corresponding to the first block of primal variables  inexactly. We established that the iterates of our method globally converge to a critical point of the original problem, and we derived convergence rates to an $\epsilon$-first-order  optimal solution, along with improved convergence rates under the KL condition. Furthermore, the numerical experiments have demonstrated the effectiveness of our proposed algorithm in solving nonlinear MPC \redd{and  matrix factorization} problems. Our work could be extended by exploring the distributed case, which could involve the development of a coordinate descent ADMM algorithm. 

\section*{Data availability}
\redd{The data that support the finding of this study are available from the corresponding author upon reasonable request.}

\section*{Conflict of interest}
The authors declare that they have no conflict of interest.

\section*{Funding}
The research leading to these results has received funding from: the European Union’s Horizon 2020 research and innovation programme under the Marie Skłodowska-Curie grant agreement No. 953348;  UEFISCDI, Romania, PN-III-P4-PCE-2021-0720, under project L2O-MOC, nr. 70/2022.



\end{document}